\documentclass{qam-l}
\usepackage{amssymb,mathrsfs,graphicx,enumerate,subfigure}
\usepackage{amsmath,amsfonts,amssymb,amscd,amsthm,bbm}
\numberwithin{figure}{section}
\usepackage{tikz}
\tikzset{
block/.style={
  draw, 
  rectangle, 
  minimum height=1.5cm, 
  minimum width=3cm, align=center
  }, 
line/.style={->,>=latex'}
}
\usepackage[retainorgcmds]{IEEEtrantools}
\usepackage{colortbl}
\usepackage{subfigure,here}

\newtheorem{theorem}{Theorem}[section]
\newtheorem{lemma}[theorem]{Lemma}

\theoremstyle{definition}
\newtheorem{definition}[theorem]{Definition}

\newtheorem{proposition}[theorem]{Proposition}

\theoremstyle{remark}
\newtheorem{remark}[theorem]{Remark}

\numberwithin{equation}{section}

\newcommand{\bbr}{\mathbb R}

\begin{document}

\title[Emergence behaviors of the thermodynamic Cucker-Smale model]{Emergence of a periodically rotating one-point cluster in a thermodynamic Cucker-Smale ensemble}


\author[Cho]{Hangjun Cho}
\address[Hangjun Cho]{ Department of Mathematical Sciences and Research Institute of Mathematics, Seoul National University, Seoul 08826, Republic of Korea} 
\email{hj2math@snu.ac.kr}
\author[Du]{Linglong Du}
\address[Linglong Du]{Department of Applied Mathematics,
 Donghua University, Shanghai, P.R. China}
 \email{matdl@dhu.edu.cn}
\thanks{Corresponding author: Linglong Du.}

\author[Ha]{Seung-Yeal Ha}
\address[Seung-Yeal Ha]{Department of Mathematical Sciences and Research Institute of Mathematics, Seoul National University, Seoul 08826 and 
Korea Institute for Advanced Study, Hoegiro 85, Seoul, 02455, Republic of Korea} 
\email{syha@snu.ac.kr}
\thanks{The work of L. Du is supported by Natural Science Foundation of China (No.12001097 and 12171082),  and  China Scholarship Council. This work was completed while the second author visited HYKE-Hwarang Lab, Seoul National University. She would like to thank Professor Ha for his hospitality during the academic year of 2020. The work of S.-Y. Ha is supported by National Research Foundation of Korea (NRF-2020R1A2C3A01003881).}


\date{}

\dedicatory{This paper has been accepted for Quarterly of Appl. Math. 
by  Professor Constantine Dafermos.}

\begin{abstract}
We study emergent behaviors of thermomechanical Cucker-Smale (TCS) ensemble confined in a harmonic potential field. In the absence of external force field, emergent dynamics of TCS particles has been extensively studied recently under various frameworks formulated in terms of initial configuration, system parameters and network topologies. Moreover, the TCS model does not exhibit rotating motions in the absence of an external force field. In this paper, we show the emergence of periodically rotating one-point cluster for the TCS model in a harmonic potential field using  elementary energy estimates and continuity argument. We also provide several numerical simulations and compare them with analytical results. 
\end{abstract}

\maketitle


\bibliographystyle{amsplain}

\section{Introduction} \label{sec:1}
\setcounter{equation}{0}
Collective behaviors of classical and quantum many-body systems are often observed in nature, e.g., aggregation of bacteria, flashing of fireflies, flocking of birds and schooling of fish,  \cite{A-B-P, A-B, B-B, D-M, M-T-0, Pe, P-R-K, T-T, T-B, VZ, Wi1}. While each particle follows their own dynamics, overall dynamics of the whole system is able to exhibit  coherent phenomena in the absence of a leader, and allows  to oder emerge from disordered configurations.  To model the emergent phenomena, several phenomenological models were proposed in literature. Among them, we are mainly interested in the Cucker-Smale (CS) type models \cite{CS1, CS2, D-M-1, H-T} describing flocking phenomena of self-propelled particle systems.  The CS model was first introduced  by Cucker and Smale \cite{CS1,CS2}, and several natural extensions of the CS model were also proposed to incorporate several aspects, to name a few,
interactions with noisy environments \cite{AH,E,HL},  normalized interactions \cite{M-T}, systems in  temperature field \cite{H-R, HKR}, systems in different forcing fields  \cite{C-D,HH,ST}, the interactions with neighboring fluids in a temperature field  \cite{CH1,CH2}. 

In what follows, we are interested in the dynamics of the TCS model in a harmonic potential field. To set the stage, let $\mathbf{x}_{\alpha}, \mathbf{v}_{\alpha}$ and $T_{\alpha}$ be the position,  velocity and temperature of the $\alpha$-th particle with unit mass in $\mathbb{R}^d$, $\alpha =1,2,\dots, n$, respectively. When the ensemble of TCS particles is under the effect of an external potential force, the temporal dynamics of the thermomechanical observables $({\mathbf x}_\alpha, {\mathbf v}_\alpha, T_\alpha)$ is governed by the following system of ordinary differential equations:
\begin{align}
\begin{aligned} \label{TCSPO}
&\frac{d\mathbf{x}_\alpha}{dt} = \mathbf{v}_\alpha, \quad t > 0, \quad \alpha = 1, \cdots, n, \\
&\frac{d\mathbf{v}_\alpha}{dt} = \frac{\kappa_1}{n} \sum_{\beta=1}^{n} \phi_{\alpha \beta} \Big( \frac{\mathbf{v}_\beta - \mathbf{v}_c}{T_\beta} -  \frac{\mathbf{v}_\alpha - \mathbf{v}_c}{T_\alpha} \Big)- \nabla_\mathbf{x} V(\mathbf{x}_\alpha), \\
&\frac{d}{dt}\!\!\left(T_\alpha + {{\frac{1}{2}  {v}_\alpha^2}}\right)\!=\! \frac{\kappa_2}{n}  \sum_{\beta=1}^{n} \zeta_{\alpha \beta} \Big( \frac{1}{T_\alpha} - \frac{1}{T_\beta}   \Big) +  \frac{\kappa_1}{n} \sum_{\beta=1}^{n} \phi_{\alpha \beta} \Big( \frac{\mathbf{v}_\beta - \mathbf{v}_c}{T_\beta} -  \frac{\mathbf{v}_\alpha - \mathbf{v}_c}{T_\alpha} \Big) \cdot \mathbf{v}_c,   
\end{aligned}   
\end{align}
where $V = V(|\mathbf x|)$ is the one-body potential, $ {v}_\alpha^2 := |\mathbf{v}_\alpha|^2$, $\kappa_1$ and $\kappa_2$ denote positive coupling strengths, ${\mathbf v}_c$ is the average velocity defined in \eqref{A-0}.

In this paper, we are interested in the collective behaviors of TCS particles in a harmonic potential field with $V(\mathbf{x}) = |\mathbf{x}|^2/2$. In this situation, system \eqref{TCSPO}  becomes
\begin{align}
\begin{aligned} \label{TCSH}
&\frac{d\mathbf{x}_\alpha}{dt} = \mathbf{v}_\alpha, \quad t > 0, \quad \alpha = 1, \cdots, n, \\
&\frac{d\mathbf{v}_\alpha}{dt} = \frac{\kappa_1}{n} \sum_{\beta=1}^{n} \phi_{\alpha \beta} \Big( \frac{\mathbf{v}_\beta - \mathbf{v}_c}{T_\beta} -  \frac{\mathbf{v}_\alpha - \mathbf{v}_c}{T_\alpha} \Big)-\mathbf{x}_\alpha, \\
&\frac{d}{dt}\!\left(T_\alpha\!+\!{{\frac{1}{2}{v}_\alpha^2}}\right)\!=\!\frac{\kappa_2}{n}  \sum_{\beta=1}^{n} \zeta_{\alpha \beta} \Big( \frac{1}{T_\alpha} - \frac{1}{T_\beta}   \Big) +  \frac{\kappa_1}{n} \sum_{\beta=1}^{n} \phi_{\alpha \beta} \Big( \frac{\mathbf{v}_\beta - \mathbf{v}_c}{T_\beta} -  \frac{\mathbf{v}_\alpha - \mathbf{v}_c}{T_\alpha} \Big) \cdot \mathbf{v}_c.
\end{aligned}   
\end{align}
Here the adjacent matrices $(\phi_{\alpha \beta})$ and $(\zeta_{\alpha \beta})$ in (\ref{TCSH}) are assumed to be dependent on the spatial difference between particles:
\begin{equation} \label{CA}
\phi_{\alpha \beta} = \phi(|\mathbf{x}_\alpha - \mathbf{x}_\beta|) \quad \text{and} \quad \zeta_{\alpha \beta}  = \zeta(|\mathbf{x}_\alpha - \mathbf{x}_\beta|),
\end{equation}
and the corresponding base functions $\phi$ and $\zeta$ are also assumed to satisfy the Lipschitz continuity and monotonicity:
\begin{align}
\begin{aligned} \label{comm}
& \phi \in \mbox{Lip}(\bbr_+; \bbr_+), \quad (\phi(r_2) - \phi(r_1)) (r_2 - r_1) \leq 0, \quad r_1, r_2 \geq 0; \\
& \zeta \in \mbox{Lip}(\bbr_+; \bbr_+), \quad (\zeta(r_2) - \zeta(r_1)) (r_2 - r_1) \leq 0, \quad r_1, r_2 \geq 0.
\end{aligned}
\end{align}

A global well-posedness of \eqref{TCSH} - \eqref{comm} can be guaranteed by the standard Cauchy-Lipschitz theory together with a priori boundedness of observables. In the absence of external force field, the emergent dynamics of  the TCS model has been extensively studied in a series of works \cite{CH1, CH2, D-H-K, H-R, HKR} from various points of views, e.g., a uniform boundedness of spatial variation, exponential decay of velocities and temperatures. These works mostly dealt with sufficient frameworks leading to the mono-cluster flocking in which all particles move with the same constant velocity asymptotically. In this paper, we are interested in the following simple question:
\begin{center}
``{\it What are the dynamical ramifications on the emergent dynamics by incorporating a harmonic potential field?} "
\end{center}
\vspace{0.2cm}
Throughout the paper, we exploit the above question analytically, i.e., we look for a sufficient framework leading to the collective dynamics for system \eqref{TCSH}. For this, we introduce center-of-mass and fluctuations around them:~for a configuration $\{ (\mathbf{x}_{\alpha}, \mathbf{v}_\alpha, T_\alpha)\}$, 
\begin{equation}
\begin{cases} \label{A-0}
 & \displaystyle  \mathbf{x}_c(t) :=\frac{1}{n}\sum_{\alpha = 1}^{n} \mathbf{x}_{\alpha}(t), \quad \mathbf{v}_c(t) :=\frac{1}{n}\sum_{\alpha = 1}^{n} \mathbf{v}_{\alpha}(t), \quad T_c(t) := \frac{1}{n} \sum_{\alpha=1}^{n} T_\alpha(t), \\
 & \displaystyle T^{\infty}(t) := -\frac{1}{2}  |{\mathbf{v}}_c(t)|^2 + T_c(0) +\frac{1}{2n}\sum\limits_{\alpha=1}^n |{\mathbf{v}}_{\alpha}(0)|^2, \\
 & \displaystyle T_m(0) :=-(|\mathbf{v}_c(0)|^2+|\mathbf{x}_c(0)|^2)+T_c(0) +\frac{1}{2n} \sum\limits_{\alpha=1}^n |{\mathbf{v}}_{\alpha}(0)|^2>0, \\
&  \displaystyle T_M(0) :=T_c(0) +\frac{1}{2n} \sum\limits_{\alpha=1}^n |{\mathbf{v}}_{\alpha}(0)|^2>0, \\
 & \displaystyle  \hat{\mathbf{x}}_{\alpha} := \mathbf{x}_\alpha - \mathbf{x}_c , \quad 
\hat{\mathbf{v}}_{\alpha} := \mathbf{v}_\alpha - \mathbf{v}_c , \quad \hat{T}_\alpha := T_\alpha -T^{\infty}, \quad \alpha = 1, \cdots, n. 
 \end{cases}
 \end{equation}
 
Note that $T_m(0)$ and $T_M(0)$ are determined only  by initial data, and in fact, they will provide lower and upper bounds for temperatures.  For the fixed center-of-mass coordinate $({\mathbf x}_c(0), {\mathbf v}_c(0)) = (\mathbf{0}, \mathbf{0})$, one has 
\[ T^{\infty}(t) = T_m(0) = T_M(0), \quad \forall~t \geq 0.    \]
From now on, for the simplicity of notation, we suppress $0$ and set 
\[    T_m := T_m(0), \quad T_M := T_M(0). \]
 Then, it follows from \eqref{A-0} and the explicit formula \eqref{B-0} for ${\mathbf v}_c$ that 
 \begin{equation} \label{A-0-0}
  \sum\limits_{\alpha=1}^n \hat{\mathbf{x}}_{\alpha}(t)=\sum\limits_{\alpha=1}^n \hat{\mathbf{v}}_{\alpha}(t)=\mathbf{0} \quad \mbox{and} \quad T_m \leq T^{\infty}(t) \leq T_M, \quad \forall~t \geq 0.
  \end{equation}
Thanks to the symmetry of communication weights in \eqref{CA},  the averages for position and velocity satisfy the harmonic oscillator equations:
\begin{equation} \label{A-1}
\frac{d\mathbf{x}_c}{dt} =\mathbf{v}_c, \quad  \frac{d\mathbf{v}_c}{dt} =-\mathbf{x}_c.
\end{equation}
It is easy to see that all nontrivial solutions to \eqref{A-1} are closed orbits (see Lemma 2.1). Thus the issue for the collective dynamics is whether the fluctuations $\{ (\hat{\mathbf{x}}_{\alpha}, \hat{\mathbf{v}}_{\alpha},  \hat{T}_\alpha) \}$ vanish asymptotically or not. If so, under what conditions, can we guarantee the asymptotic vanishing of fluctuations  so that the whole particle system behaves like a single particle with a closed trajectory? This is the issue that we would like to address in this paper. \newline

Next, we briefly discuss our main results and  strategy. As noted in earlier works \cite{H-R, HKR}, temperatures appear in the denominators of $\eqref{TCSH}_2 -\eqref{TCSH}_3$,  thus it is necessary to guarantee the positivity of temperatures to construct a global solution satisfying desired emergent estimates. For the derivation of desired emergent estimates, we first assume that initial data satisfy some admissible condition, and temperatures are away from zero at least in some small-time interval (a priori assumption on temperatures below), i.e., for some positive constant $\varepsilon_0$, $\forall~\delta>3$ and $\tau \in (0, \infty]$,
\begin{equation} \label{A-2}
T_m \ge \delta\varepsilon_0, \quad \sup_{0 \leq t < \tau}  |{\hat T}_{\alpha}(t)| \leq \varepsilon_0,
\end{equation}
we derive a dissipative energy estimate for $t\in[0,\tau)$ (see Proposition \ref{P3.1}):
\begin{align}
\begin{aligned} \label{A-3}
&\frac{d}{dt}\sum\limits_{\alpha = 1}^{n}\left(\frac{1}{2}|\hat{\mathbf{x}}_\alpha|^2 + \frac{1}{2}|\hat{\mathbf{v}}_\alpha|^2 + \varepsilon \mathbf{\hat{x}}_\alpha\cdot \mathbf{\hat{v}}_\alpha \right)\\
&\le \Big(-2\lambda \phi(\sqrt{2}\mathcal{X})+\varepsilon\gamma+\frac{\kappa_1\varepsilon_0 \phi(0)}{ (T_m - \varepsilon_0)^2 } \Big) \sum\limits_{\alpha = 1}^{n} | \hat{\mathbf{v}}_\alpha|^2-\frac{\varepsilon}{2}\sum\limits_{\alpha = 1}^{n} | \hat{\mathbf{x}}_\alpha|^2,
\end{aligned}
\end{align}
where $\varepsilon, \lambda$ and $\gamma$ are positive constants, $\mathcal{X}$ is  $\ell^2$-norm of spatial fluctuation defined in \eqref{C-1}. Once we can establish the dissipative differential inequality \eqref{A-3}, we can use  a bootstrap argument together with \eqref{A-3} to derive the desired emergent estimate for mechanical observables (see Proposition \ref{L4.1}):
\begin{equation} \label{A-4}
\sum_{\alpha = 1}^{n} |\mathbf{x}_\alpha(t) - \mathbf{x}_c(t) |^2  +  \sum_{\alpha = 1}^{n} |\mathbf{v}_\alpha(t) - \mathbf{v}_c(t) |^2 \le  {\mathcal O}(1)e^{-\frac{2\varepsilon}{3}t}, \quad t\in[0,\tau).
\end{equation}
On the other hand, for temperature homogenization, we also derive Gr\"onwall's type differential inequality for temperature fluctuations under the same a priori condition \eqref{A-2} (see Proposition \ref{P4.2a}):
\[ \frac{d}{dt} \sum_{\alpha = 1}^{n} |T_\alpha(t) -T^{\infty}(t)|^2 \leq -C_0 \sum_{\alpha = 1}^{n} |T_\alpha(t) -T^{\infty}(t)|^2 + {\mathcal O}(1) e^{-\frac{2\varepsilon}{3}t}, \quad t\in[0,\tau),
\] where $C_0$ is a positive constant.
This again yields an exponential decay of temperature fluctuations (see Proposition \ref{P4.2}): There exists a positive constant $\Lambda > 0$ such that 
\begin{equation} \label{A-5}
\sum_{\alpha = 1}^{n} |T_\alpha(t) -T^{\infty}(t)|^2 \le {\mathcal O}(1) e^{-\Lambda t}, \quad t\in[0,\tau).
\end{equation}
Finally, we show that under suitable assumptions on initial data, the a priori condition \eqref{A-2}  holds for all time $t$ and obtain desired emergent estimates \eqref{A-4} and \eqref{A-5} (see Theorem \ref{T5.1}). We refer to \cite{ST} for related results to the hydrodynamic CS model with an external potential force (see Remark \ref{R5.1}). \newline

The rest of this paper is organized as follows. In Section \ref{sec:2}, we study basic estimates for system \eqref{TCSH} and formulations for fluctuation dynamics. In Section \ref{sec:3}, we derive an energy estimate. In Section \ref{sec:4}, we provide emergence of  periodically rotating one-point cluster  in a small-time interval under a priori condition \eqref{A-2}. In Section \ref{sec:5},  we remove the a priori condition on temperatures by imposing suitable conditions on the initial data and continuity argument, and derive  desired emergent estimates for all time. In Section \ref{sec:6}, we provide several numerical simulations and compare them with analytical results  in Section \ref{sec:5}. Finally, Section \ref{sec:7} is devoted to a brief summary of our main results and some remaining issues for a future work.

\section{Preliminaries} \label{sec:2}
\setcounter{equation}{0}
In this section, we study basic estimates for \eqref{TCSH} such as conservation laws, and then derive a dynamical system for the fluctuations \eqref{A-0}.

\subsection{Basic estimates} \label{sec:2.1} 
For $n$-particle configuration  $\{ (\mathbf{x}_\alpha, \mathbf{v}_\alpha, T_\alpha) \}$, we set 
\begin{equation*} \label{B-0-0}
{\mathcal E} := \sum_{\alpha = 1}^{n}  \Big( T_{\alpha} + \frac{1}{2} |\mathbf{v}_{\alpha}|^2 \Big).
\end{equation*}
As in \cite{H-R}, one has the temporal evolution of the above functional. 
\begin{lemma} \label{L2.1}
\emph{\cite{H-R}}
    For $\tau \in (0, \infty]$, let $\{(\mathbf{x}_{\alpha}, \mathbf{v}_{\alpha}, T_{\alpha})\}$ be a smooth solution in a time-interval $[0, \tau)$ to system \eqref{TCSH} - \eqref{comm} with initial data satisfying
    \[ | {\mathbf x}_c(0) | < \infty, \quad| {\mathbf v}_c(0) | < \infty, \quad {\mathcal E}(0) < \infty.     \]
    Then, the following estimates hold:
    \begin{enumerate}
       \item
Center-of-masses ${\mathbf x}_c$ and ${\mathbf v}_c$ are given by the following explicit formula:
 \begin{equation} \label{B-0}
\mathbf{x}_c(t) = (\cos t) \mathbf{x}_c(0) + (\sin t) \mathbf{v}_c(0), \quad 
\mathbf{v}_c(t) = -(\sin t) \mathbf{x}_c(0) + (\cos t) \mathbf{v}_c(0).
\end{equation}
        \item
        The total energy satisfies 
        \[
        {\mathcal E}(t) = {\mathcal E}(0), \quad   t \in (0, \tau).
        \]
       \end{enumerate}
\end{lemma}
\begin{proof}     
\noindent (1)  We sum $\eqref{TCSH}_2$ over all $\alpha$ to get 
\begin{equation} \label{B-1}
\frac{d}{dt} \sum_{\alpha=1}^n \mathbf{v}_\alpha = \frac{\kappa_1}{n} \sum_{\alpha, \beta=1}^{n} \phi_{\alpha \beta} \Big( \frac{\mathbf{v}_\beta - \mathbf{v}_c}{T_\beta} -  \frac{\mathbf{v}_\alpha - \mathbf{v}_c}{T_\alpha} \Big)- \sum_{\alpha = 1}^{n} \mathbf{x}_\alpha. 
\end{equation}
Now, we use index exchange transformation $\alpha \longleftrightarrow \beta$ and symmetry of $\phi_{\alpha \beta} = \phi_{\beta \alpha}$ to see
\[
 \frac{\kappa_1}{n} \sum_{\alpha, \beta=1}^{n} \phi_{\alpha \beta} \Big( \frac{\mathbf{v}_\beta - \mathbf{v}_c}{T_\beta} -  \frac{\mathbf{v}_\alpha - \mathbf{v}_c}{T_\alpha} \Big) = 
 - \frac{\kappa_1}{n} \sum_{\alpha, \beta=1}^{n} \phi_{\alpha \beta} \Big( \frac{\mathbf{v}_\beta - \mathbf{v}_c}{T_\beta} -  \frac{\mathbf{v}_\alpha - \mathbf{v}_c}{T_\alpha} \Big).
\]
Thus, one has 
\begin{equation} \label{B-2}
\frac{\kappa_1}{n} \sum_{\alpha, \beta=1}^{n} \phi_{\alpha \beta} \Big( \frac{\mathbf{v}_\beta - \mathbf{v}_c}{T_\beta} -  \frac{\mathbf{v}_\alpha - \mathbf{v}_c}{T_\alpha} \Big)  = \mathbf{0}. 
\end{equation}
By \eqref{B-1}-\eqref{B-2}  and  $\eqref{TCSH}_1$, $({\mathbf x}_c, {\mathbf v}_c)$ satisfies 
\[
\frac{d{\mathbf x}_c}{dt} = {\mathbf v}_c, \quad  \frac{d{\mathbf v}_c}{dt} =-{\mathbf x}_c,
\]
which yields the desired estimates \eqref{B-0}. \newline

\noindent (2)~Again, we take a sum $\eqref{TCSH}_3$ over all $\alpha$ and use the symmetry of $\zeta_{\alpha \beta}$ and $\phi_{\alpha \beta}$ to derive the desired estimate:
\[ \frac{d}{dt} \sum_{\alpha} \left(T_\alpha + {{\frac{1}{2}  {v}_\alpha^2}}\right) = 0.    \]
\end{proof}
\begin{remark} 
Note that  the explicit relation  \eqref{B-0} implies
\[ {\mathbf x}_c(0) = \mathbf{0}, \quad {\mathbf v}_c(0)= \mathbf{0}  \quad \Longrightarrow \quad {\mathbf x}_c(t) = \mathbf{0}, \quad {\mathbf v}_c(t)= \mathbf{0}, \quad t \geq 0.  \]
\end{remark}
Next, we study the Gr\"onwall's type lemma to be used later.
\begin{lemma}\label{L2.2}
Let $y: \bbr_+ \to \bbr$ be a nonnegative Lipschitz function satisfying
\[ 
\begin{cases}
\displaystyle y^{\prime} \le-c_1 y +c_2 e^{-c_3 t}, \quad \mbox{a.e.}~t > 0, \\
\displaystyle y(0) = y^0, 
\end{cases}
\]
where $c_1, c_2$ and $c_3$ are positive constants. Then, one has
\[y(t) \leq y^0 e^{-c_1 t}+\frac{c_2}{c_1-c_3}(e^{-c_3t}-e^{-c_1t}), \quad t\ge 0.\]
\end{lemma}
\begin{proof} By Gr\"onwall's inequality, one has
\begin{align*}
\begin{aligned}
y(t)&\le e^{\int_0^t -c_1ds}\Big[ y^0 +\int_0^t c_2 e^{-c_3 s}e^{-\int_0^s -c_1dr}ds\Big]\\
&\le e^{-c_1 t}\Big[ y^0+ c_2 \int_0^te^{(c_1-c_3)s}ds\Big] \\
&=e^{-c_1 t}\Big[ y^0+\frac{c_2}{c_1-c_3}(e^{(c_1-c_3)t}-1)\Big].
\end{aligned}
\end{align*}
This yields the desired estimate.
\end{proof}

\subsection{A dynamical system for fluctuations} \label{sec:2.2}
In this subsection, we present a dynamical system for fluctuations \eqref{A-0}, which will be used crucially in the following sections. \newline

First, we introduce an emergence of periodically rotating one-point cluster in the following definition. 
\begin{definition} 
 Let $\{({\mathbf x}_\alpha, {\mathbf v}_\alpha, T_\alpha) \}$ be a solution to \eqref{TCSH}. Then, the configuration approaches to the  periodically rotating one-point cluster (PROC) asymptotically if the following condition holds:
\[  \lim_{t \to \infty} \max_{\alpha} \Big(|{\mathbf x}_\alpha(t) -{\mathbf x}_c(t) | + |{\mathbf v}_\alpha(t) -{\mathbf v}_c(t) | +  |T_\alpha(t) - T^{\infty}(t) | \Big) = 0.\]
Here $({\mathbf x}_c(t),{\mathbf v}_c(t), T^{\infty}(t))$ is  a periodic state. 
 \end{definition}
Note that fluctuations $\{ (\hat{\mathbf{x}}_{\alpha},\hat{\mathbf{v}}_{\alpha}, \hat{T}_\alpha) \}$ defined in \eqref{A-0}  satisfy
\begin{equation} \label{B-5}
\begin{cases}
\displaystyle \frac{d\hat{\mathbf{x}}_\alpha}{dt} = \hat{\mathbf{v}}_\alpha,  \quad t > 0, \quad \alpha = 1, \cdots, n, \\
\displaystyle \frac{d\hat{\mathbf{v}}_\alpha}{dt} =\frac{\kappa_1}{n} \sum_{\beta=1}^{n} \phi(|\hat{\mathbf{x}}_\alpha - \hat{\mathbf{x}}_\beta|) \Big( \frac{\hat{\mathbf{v}}_\beta}{{\hat T}_\beta + T^{\infty}} -  \frac{\hat{\mathbf{v}}_\alpha }{{\hat T}_\alpha + T^{\infty}} \Big)- \hat{\mathbf{x}}_\alpha, \\
\displaystyle \frac{d}{dt} (\hat T_\alpha+\frac{1}{2}|\hat{\mathbf{v}}_{\alpha}|^2) = \frac{\kappa_2}{n} \sum_{\beta=1}^{n} \zeta(|\hat{\mathbf{x}}_\alpha - \hat{\mathbf{x}}_\beta|)  \Big( \frac{1}{{\hat T}_\alpha + T^{\infty}} - \frac{1}{{\hat T}_\beta + T^{\infty}}\Big) \\
\displaystyle \hspace{3.3cm} +~\mathbf{x}_c\cdot\hat{\mathbf{v}}_{\alpha}+\mathbf{v}_c\cdot\hat{\mathbf{x}}_{\alpha}.
\end{cases}
\end{equation}

Before we close this section, we briefly outline our strategy in three steps, which will be  illustrated in the following three sections separately. \newline
\begin{itemize}
\item
Step A (Derivation of a differential inequality for energy functional under a priori condition \eqref{A-2}): ~We set 
\begin{equation} \label{L} {\mathcal L}(\hat{\mathbf x}, \hat{\mathbf v}) : = \sum\limits_{\alpha = 1}^{n} \left(\frac{1}{2}|\hat{\mathbf{x}}_\alpha|^2 + \frac{1}{2}|\hat{\mathbf{v}}_\alpha|^2 + \varepsilon \mathbf{\hat{x}}_\alpha\cdot \mathbf{\hat{v}}_\alpha\right). 
\end{equation}
Then, it is equivalent to the standard energy functional $\sum\limits_{\alpha = 1}^{n} \Big(\frac{1}{2}|\hat{\mathbf{x}}_\alpha|^2 + \frac{1}{2}|\hat{\mathbf{v}}_\alpha|^2 \Big)$, i.e., there exists a generic positive constant $C$ such that 
\begin{equation*} \label{B-6}
\frac{1}{C} \sum\limits_{\alpha = 1}^{n} \Big(\frac{1}{2}|\hat{\mathbf{x}}_\alpha|^2 + \frac{1}{2}|\hat{\mathbf{v}}_\alpha|^2 \Big) \leq  {\mathcal L}(\hat{\mathbf x},\hat{\mathbf v}) \leq C \sum\limits_{\alpha = 1}^{n} \Big(\frac{1}{2}|\hat{\mathbf{x}}_\alpha|^2 + \frac{1}{2}|\hat{\mathbf{v}}_\alpha|^2 \Big).
\end{equation*}
Under the a priori condition \eqref{A-2}, we derive a differential inequality for ${\mathcal L}$:
\begin{equation} \label{B-7}
\frac{d}{dt} {\mathcal L} (\hat{\mathbf x}, \hat{\mathbf v})  \leq \Big(-2\lambda \phi(\sqrt{2}\mathcal{X})+\varepsilon\gamma+\frac{\kappa_1\varepsilon_0 \phi(0)}{ (T_m - \varepsilon_0)^2 } \Big) \sum\limits_{\alpha = 1}^{n} | \hat{\mathbf{v}}_\alpha|^2-\frac{\varepsilon}{2}\sum\limits_{\alpha = 1}^{n} | \hat{\mathbf{x}}_\alpha|^2.
\end{equation}

\vspace{0.2cm}

\item
Step B (Derivation of emergence for PROC under a priori condition \eqref{A-2}):~First, we use the differential inequality \eqref{B-7} to see that at least in a small-time interval $[0, \tau)$, spatial fluctuations are sufficiently small
\[   \sum_{\alpha = 1}^{n} |\mathbf{x}_\alpha(t) - \mathbf{x}_c(t) |^2 \ll 1, \quad \forall~t \in [0, \tau). \]
By bootstrapping the above rough estimate in the differential inequality \eqref{B-7}, we get an exponential decay estimate \eqref{A-4} for fluctuations. Fluctuations for temperatures can be dealt with analogously by deriving a differential inequality for them, and we obtain the exponential decay to  zero for  temperature fluctuations.
 
\vspace{0.2cm}

\item
Step C (Removal of a priori condition  \eqref{A-2}): ~Finally, we replace the a priori condition for temperatures by admissible conditions for initial data and then use the continuity argument, we show that the maximal time span $\tau$ in Step B can be extended to infinity. Thus, fluctuations for mechanical observables and temperature decay to zero exponentially fast, and then the overall asymptotic dynamics of \eqref{TCSH} is governed by the harmonic oscillator equations whose nontrivial solutions are periodic orbits. 
\end{itemize}
\section{Derivation of a dissipative energy estimate} \label{sec:3}
\setcounter{equation}{0}
In this section, we derive a differential inequality \eqref{B-7} for the energy functional ${\mathcal L}(\hat{\mathbf x}, \hat{\mathbf v})$ under the a priori condition \eqref{A-2}. For this, we perform estimates on the time-evolution of the following thee components in ${\mathcal L}(\hat{\mathbf x}, \hat{\mathbf v})$:
\[  \sum\limits_{\alpha = 1}^{n}\frac{1}{2}|\hat{\mathbf{x}}_\alpha|^2, \quad  \sum_{\alpha =1}^{n} \frac{1}{2} |\hat{\mathbf{v}}_\alpha|^2, \quad   \sum_{\alpha =1}^{n}   \mathbf{\hat{x}}_\alpha\cdot \mathbf{\hat{v}}_\alpha.       \]
We set the time-dependent $\ell^2$-norms of fluctuations \eqref{A-0} as 
\begin{equation} \label{C-1}
{\mathcal X} := \Big( \sum\limits_{\alpha = 1}^{n} |\hat{\mathbf{x}}_\alpha|^2 \Big)^{\frac{1}{2}}, \quad 
{\mathcal V} := \Big( \sum\limits_{\alpha = 1}^{n} |\hat{\mathbf{v}}_\alpha|^2 \Big)^{\frac{1}{2}}, \quad
{\mathcal T} := \Big( \sum\limits_{\alpha = 1}^{n} |\hat{T}_\alpha |^2 \Big)^{\frac{1}{2}}.
 \end{equation}
 
 Next we provide estimates for the component functionals  in  ${\mathcal L}(\hat{\mathbf x}, \hat{\mathbf v})$ one by one.
 
\noindent $\bullet$~Case A.1~(Estimate on $\sum\limits_{\alpha = 1}^{n}\frac{1}{2}|\hat{\mathbf{x}}_\alpha|^2$):~We take an inner product $\mathbf{\hat{x}}_{\alpha}$ with $\eqref{B-5}_1$  and sum the resulting relation over $\alpha$ to get 
  \begin{equation} \label{C-2}
\frac{d}{dt} \left(\sum\limits_{\alpha = 1}^{n}\frac{1}{2}|\hat{\mathbf{x}}_\alpha|^2\right) = \sum\limits_{\alpha = 1}^{n} \mathbf{\hat{x}}_\alpha\cdot \mathbf{\hat{v}}_\alpha.
    \end{equation}

\vspace{0.2cm}

\noindent $\bullet$~Case A.2~(Estimate on $\sum\limits_{\alpha = 1}^{n}\frac{1}{2}|\hat{\mathbf{v}}_\alpha|^2$):~We take an inner product $\hat{\mathbf{v}}_{\alpha}$ with $\eqref{B-5}_2$  and sum up the resulting relation over $\alpha$ to obtain
\begin{align}
\begin{aligned} \label{C-3}
&\frac{d}{dt} \sum_{\alpha =1}^{n}  \frac{|\hat{\mathbf{v}}_\alpha|^2}{2} =  \frac{\kappa_1}{n} \sum\limits_{\alpha, \beta=1}^{n} \phi(|\hat{\mathbf{x}}_\alpha - \hat{\mathbf{x}}_\beta|)\hat{\mathbf{v}}_{\alpha} \cdot \Big( \frac{\hat{\mathbf{v}}_\beta}{{\hat T}_\beta + T^{\infty}} -  \frac{\hat{\mathbf{v}}_\alpha}{{\hat T}_\alpha + T^{\infty}} \Big) -\sum\limits_{\alpha = 1}^{n} \mathbf{\hat{x}}_\alpha\cdot \mathbf{\hat{v}}_\alpha\\
&
 =-\frac{\kappa_1}{n} \sum\limits_{\alpha, \beta=1}^{n} \phi(|\hat{\mathbf{x}}_\alpha - \hat{\mathbf{x}}_\beta|)
\hat{\mathbf{v}}_{\beta} \cdot \Big( \frac{\hat{\mathbf{v}}_\beta}{{\hat T}_\beta + T^{\infty}} -  \frac{\hat{\mathbf{v}}_\alpha}{{\hat T}_\alpha + T^{\infty}} \Big) -\sum\limits_{\alpha = 1}^{n} \mathbf{\hat{x}}_\alpha\cdot \mathbf{\hat{v}}_\alpha\\
&
=\frac{\kappa_1}{2n} \sum\limits_{\alpha, \beta=1}^{n} \phi(|\hat{\mathbf{x}}_\alpha - \hat{\mathbf{x}}_\beta|) (\hat{\mathbf{v}}_{\alpha} - \hat{\mathbf{v}}_{\beta}) \cdot \Big( \frac{\hat{\mathbf{v}}_\beta}{{\hat T}_\beta + T^{\infty}} -  \frac{\hat{\mathbf{v}}_\alpha}{{\hat T}_\alpha + T^{\infty}} \Big)-\sum\limits_{\alpha = 1}^{n} \mathbf{\hat{x}}_\alpha\cdot \mathbf{\hat{v}}_\alpha \\
&
= -\frac{\kappa_1}{2n} \sum\limits_{\alpha, \beta=1}^{n} \phi(|\hat{\mathbf{x}}_\alpha - \hat{\mathbf{x}}_\beta|)\frac{1}{ {\hat T}_{\beta} + T^{\infty}  } \
    |\hat{\mathbf{v}}_{\alpha} - \hat{\mathbf{v}}_{\beta}|^2 \\
&\hspace{0.5cm} + \frac{\kappa_1}{2n} \sum\limits_{\alpha, \beta=1}^{n} \phi(|\hat{\mathbf{x}}_\alpha - \hat{\mathbf{x}}_\beta|)
 \frac{{\hat T}_{\alpha} - {\hat T}_{\beta}}{ ({\hat T}_{\alpha} + T^{\infty}) ({\hat T}_{\beta} + T^{\infty})  } \hat{\mathbf{v}}_{\alpha} \cdot(\hat{\mathbf{v}}_{\alpha} - \hat{\mathbf{v}}_{\beta}) -\sum\limits_{\alpha = 1}^{n} \mathbf{\hat{x}}_\alpha\cdot \mathbf{\hat{v}}_\alpha \\
 &
 =:  {\mathcal I}_{11} + {\mathcal I}_{12}   -\sum\limits_{\alpha = 1}^{n} \mathbf{\hat{x}}_\alpha\cdot \mathbf{\hat{v}}_\alpha,
\end{aligned}
\end{align}
where we used the relation:
\begin{equation} \label{C-4}
 \frac{\mathbf{\hat{{v}}}_\beta}{{\hat T}_\beta + T^{\infty}} -  \frac{\hat{\mathbf{v}}_\alpha}{{\hat T}_\alpha + T^{\infty}} = \frac{1}{{\hat T}_{\beta} + T^{\infty} } 
(\hat{\mathbf{v}}_{\beta} - \hat{\mathbf{v}}_{\alpha}) +\frac{ {\hat T}_{\alpha} - {\hat T}_{\beta}}{ ({\hat T}_{\alpha} + T^{\infty}) ({\hat T}_{\beta} + T^{\infty})  }\hat{\mathbf{v}}_{\alpha}.
\end{equation}
In what follows, we use the following  estimates:
\begin{align}
\begin{aligned} \label{C-5}
& |\hat{\mathbf{x}}_\alpha -\hat{\mathbf{x}}_\beta|\le\sqrt{2(|\hat{\mathbf{x}}_\alpha|^2+|\hat{\mathbf{x}}_\beta|^2 )}\le\sqrt{2} \mathcal{X}, \\
& 0 <T_m- \varepsilon_0 \le T^{\infty}- \varepsilon_0 \leq {\hat T}_{\alpha} + T^{\infty}\leq T^{\infty}+ \varepsilon_0\le T_M+ \varepsilon_0.
\end{aligned}
\end{align}
\noindent $\diamond$~(Estimate of ${\mathcal I}_{11} $): ~We use \eqref{C-5} and \eqref{A-0-0} to obtain
\begin{align}
\begin{aligned} \label{C-7}
{\mathcal I}_{11} &\leq -\frac{\kappa_1}{2n} \sum\limits_{\alpha, \beta=1}^{n} \phi(|\hat{\mathbf{x}}_\alpha - \hat{\mathbf{x}}_\beta|) \frac{1}{ T^{\infty} + \varepsilon_0} |\hat{\mathbf{v}}_{\alpha} - \hat{\mathbf{v}}_{\beta}|^2
\le - \frac{\kappa_1\phi(\sqrt{2}\mathcal{X}) }{ T_M + \varepsilon_0}  \sum\limits_{\alpha=1}^{n} |\hat{\mathbf{v}}_\alpha|^2.
\end{aligned}
\end{align}

\noindent $\diamond$~(Estimate of ${\mathcal I}_{12} $): ~Similarly, one has 
\begin{align}
\begin{aligned} \label{C-8}
{\mathcal I}_{12} &=\frac{\kappa_1}{2n} \sum\limits_{\alpha, \beta=1}^{n} \phi(|\hat{\mathbf{x}}_\alpha - \hat{\mathbf{x}}_\beta|)
\frac{{\hat T}_{\alpha} - {\hat T}_{\beta}}{ ({\hat T}_{\alpha} + T^{\infty}) ({\hat T}_{\beta} + T^{\infty})  }  \hat{\mathbf{v}}_{\beta} \cdot(\hat{\mathbf{v}}_{\alpha} - \hat{\mathbf{v}}_{\beta})\\
&=\frac{\kappa_1}{4n} \sum\limits_{\alpha, \beta=1}^{n} \phi(|\hat{\mathbf{x}}_\alpha - \hat{\mathbf{x}}_\beta|)\
 \frac{{\hat T}_{\alpha} - {\hat T}_{\beta}}{ ({\hat T}_{\alpha} + T^{\infty}) ({\hat T}_{\beta} + T^{\infty})  } \Big( |\hat{\mathbf{v}}_{\alpha}|^2-|\hat{\mathbf{v}}_{\beta}|^2\Big)\\
   &\le \frac{2\varepsilon_0}{4n}\frac{\kappa_1\phi(0)}{ (T^{\infty} - \varepsilon_0)^2 }  \sum\limits_{\alpha, \beta=1}^{n} 
\Big( |\hat{\mathbf{v}}_{\alpha}|^2+ |\hat{\mathbf{v}}_{\beta}|^2\Big) \leq \frac{\kappa_1\varepsilon_0\phi(0)}{ (T_m - \varepsilon_0)^2 }  \sum\limits_{\alpha=1}^{n} |\hat{\mathbf{v}}_\alpha|^2.
\end{aligned}
\end{align}
In \eqref{C-3}, we combine all the estimates \eqref{C-7} and \eqref{C-8} to get 
\begin{align}
\begin{aligned} \label{C-9}
\frac{d}{dt} \sum_{\alpha =1}^{n} \frac{|\mathbf{\hat{{v}}_\alpha}|^2}{2}&\leq \left(-\frac{\kappa_1\phi(\sqrt{2}{\mathcal{X}}) }{ T_M + \varepsilon_0} +\frac{\kappa_1\varepsilon_0 \phi(0)}{ (T_m - \varepsilon_0)^2 } \right) \sum\limits_{\alpha=1}^{n} |\hat{\mathbf{v}}_\alpha|^2-\sum\limits_{\alpha = 1}^{n} \mathbf{\hat{x}}_\alpha\cdot \mathbf{\hat{v}}_\alpha.
\end{aligned}
\end{align}

\noindent $\bullet$~Case A.3~(Estimate on $\sum\limits_{\alpha = 1}^{n} \mathbf{\hat{x}}_\alpha\cdot \mathbf{\hat{v}}_\alpha$):~ We use the assumption $T_m \ge \delta\varepsilon_0$ to find
\begin{align}
\begin{aligned} \label{C-10}
&\frac{d}{dt} \left(\sum\limits_{\alpha = 1}^{n} \mathbf{\hat{x}}_\alpha\cdot \mathbf{\hat{v}}_\alpha\right) =\sum\limits_{\alpha = 1}^{n} \left( \frac{d{\hat{\mathbf{x}}}_\alpha}{dt}\cdot \hat{\mathbf{v}} _\alpha+\hat{\mathbf{x}}_\alpha\cdot  \frac{d{\hat{\mathbf{v}} }_\alpha}{dt}\right)\\
&= \sum\limits_{\alpha = 1}^{n}| \hat{\mathbf{v}}_\alpha|^2+  \frac{\kappa_1}{n}\sum_{\alpha,\beta=1} ^n\phi(|\hat{\mathbf{x}}_\alpha - \hat{\mathbf{x}}_\beta|) \hat{\mathbf{x}}_\alpha \cdot \Big( \frac{\hat{\mathbf{v}}_\beta}{T_\beta} -  \frac{\hat{\mathbf{v}}_\alpha}{T_\alpha } \Big) - \sum\limits_{\alpha=1}^{n}|\hat{\mathbf{x}}_{\alpha}|^2 \\ 
  &\le  - \sum\limits_{\alpha = 1}^{n}| \hat{\mathbf{x}}_\alpha|^2+\sum\limits_{\alpha = 1}^{n}| \hat{\mathbf{v}}_\alpha|^2+ \frac{ \kappa_1\phi(0)}{n}\sum\limits_{\alpha,\beta=1}^n \left(\frac{\kappa_1\phi(0)}{2}\left(\frac{\hat{\mathbf{v}}_\beta}{T_\beta} - \frac{\hat{\mathbf{v}}_\alpha}{T_\alpha } \right)^2 + \frac{1}{2\kappa_1\phi(0)} \hat{\mathbf{x}}_\alpha^2\right) \\
&\le  - \frac{1}{2}\sum\limits_{\alpha = 1}^{n}| \hat{\mathbf{x}}_\alpha|^2+ \sum\limits_{\alpha = 1}^{n}| \hat{\mathbf{v}}_\alpha|^2 +(\kappa_1\phi(0))^2 \left(\frac{2}{(T^{\infty} - \varepsilon_0)^2} +  \frac{4\varepsilon_0^2}{(T^{\infty} - \varepsilon_0)^4}\right)\sum\limits_{\alpha = 1}^{n}| \hat{\mathbf{v}}_\alpha|^2 \\
 &= - \frac{1}{2}\sum\limits_{\alpha = 1}^{n}| \hat{\mathbf{x}}_\alpha|^2 +\left ((\kappa_1\phi(0))^2 \left(\frac{2}{(T^{\infty} - \varepsilon_0)^2} +  \frac{4\varepsilon_0^2}{(T^{\infty} - \varepsilon_0)^4}\right)+1\right) \sum\limits_{\alpha = 1}^{n}| \hat{\mathbf{v}}_\alpha|^2\\
 &\le - \frac{1}{2}\sum\limits_{\alpha = 1}^{n}| \hat{\mathbf{x}}_\alpha|^2  +\left ((\kappa_1\phi(0))^2\frac{3}{(T_m - \varepsilon_0)^2} +1\right) \sum\limits_{\alpha = 1}^{n}| \hat{\mathbf{v}}_\alpha|^2,
\end{aligned}
\end{align}
where we used the estimates:
 \begin{align*}
 \begin{aligned}
&\sum\limits_{\alpha,\beta=1}^n \left(\frac{\hat{\mathbf{v}}_\beta}{T_\beta} - \frac{\hat{\mathbf{v}}_\alpha}{T_\alpha } \right)^2 \\
&= \sum\limits_{\alpha,\beta=1}^n \left(\frac{(\hat{\mathbf{v}}_\beta - \hat{\mathbf{v}}_\alpha)}{T_\beta} + \frac{\hat{\mathbf{v}}_\alpha}{T_\alpha  T_\beta} (\hat{T}_\alpha - \hat{T}_\beta)\right)^2 \leq  2\sum\limits_{\alpha,\beta=1}^n \left(\frac{\hat{\mathbf{v}}_\beta^2 +\hat{\mathbf{v}}_\alpha^2}{(T^{\infty} - \varepsilon_0)^2} +  \frac{4\varepsilon_0^2}{(T^{\infty} - \varepsilon_0)^4} \hat{\mathbf{v}}_\alpha^2 \right) \\
 &\le 2\sum\limits_{\alpha = 1}^{n}| \hat{\mathbf{v}}_\alpha|^2 \!\left(\frac{2n}{(T^{\infty}\! -\! \varepsilon_0)^2} + n \frac{4\varepsilon_0^2}{(T^{\infty} \!-\! \varepsilon_0)^4}\right) \!= \!2n \left(\frac{2}{(T^{\infty}\!-\! \varepsilon_0)^2} \!+\!\frac{4\varepsilon_0^2}{(T^{\infty} \!-\! \varepsilon_0)^4}\right) \sum\limits_{\alpha = 1}^{n}| \hat{\mathbf{v}}_\alpha|^2,
\end{aligned}
 \end{align*}
 and
 \[ ab\le \left(\eta a^2+\frac{b^2}{\eta}\right)/2, \quad (a\pm b)^2\le 2(a^2+b^2), \quad \sum\limits_{\alpha=1}^n \hat{\mathbf{v}}_\alpha = \mathbf{0}. \]
Finally, we combine all the estimates \eqref{C-2}, \eqref{C-9} and \eqref{C-10} to obtain a dissipative  differential inequality in next proposition. 
\begin{proposition} \label{P3.1}
 For $\tau \in (0, \infty]$, let
    $\{(\hat{\mathbf{x}}_{\alpha}, \hat{\mathbf{v}}_{\alpha}, \hat{T}_{\alpha}) \}$ be a
    solution to system  \eqref{B-5} with the initial data $(\hat{\mathbf{x}}_{\alpha}(0), \hat{\mathbf{v}}_{\alpha}(0), \hat{T}_{\alpha}(0))$ satisfying a priori condition:
    \begin{equation} \label{C-11}
      T_m \ge \delta\varepsilon_0,   \quad \sup_{0 \leq t < \tau}  |{\hat T}_{\alpha}(t)| \leq \varepsilon_0,
    \end{equation}
    where $\varepsilon_0$ is a positive constant, $\forall\delta>3$. 
    Then  for $t \in (0, \tau)$, $\forall \varepsilon>0$, one has a dissipative inequality:
    \begin{align*}
    \begin{aligned}
    & \frac{d}{dt}  \sum\limits_{\alpha = 1}^{n} \left(\frac{1}{2}|\hat{\mathbf{x}}_\alpha|^2 + \frac{1}{2}|\hat{\mathbf{v}}_\alpha|^2 + \varepsilon \mathbf{\hat{x}}_\alpha\cdot \mathbf{\hat{v}}_\alpha\right)  \\
    &\le -\left(\frac{\kappa_1\phi(\sqrt{2}{\mathcal{X}}) }{ T_M + \varepsilon_0} -\frac{\kappa_1\varepsilon_0 \phi(0)}{ (T_m - \varepsilon_0)^2 } -\varepsilon\left (\frac{3(\kappa_1\phi(0) )^2}{(T_m - \varepsilon_0)^2} +1\right)\right) \sum\limits_{\alpha = 1}^{n}| \hat{\mathbf{v}}_\alpha|^2 - \frac{\varepsilon}{2}\sum\limits_{\alpha=1}^n| \hat{\mathbf{x}}_\alpha|^2 .
    \end{aligned}
    \end{align*}
\end{proposition}
\begin{proof}
For a positive constant $\varepsilon$, we take a linear combination:
\[ \eqref{C-2} + \eqref{C-9} + \varepsilon \times \eqref{C-10} \]
to derive the desired estimate.
\end{proof}
\begin{remark} 
 Note that the a priori condition \eqref{C-11} on temperature implies the positivity and boundedness of temperatures: since 
\[ T_\alpha(t) = T^{\infty}(t) + T_\alpha(t) - T^{\infty}(t), \]
it follows from \eqref{A-0} and \eqref{A-0-0} that for $t \in [0, \tau)$,
\[ 0 < T_m - \varepsilon_0 \leq T^{\infty} - |{\hat T}_\alpha(t)| \leq T_\alpha(t) \leq T^{\infty}(t) + |{\hat T}_\alpha(t)| \leq T_M + \varepsilon_0 < \infty.  \]

\end{remark}

 \section{Decay estimates of fluctuations in a small-time interval} \label{sec:4}
 \setcounter{equation}{0}
 In this section, we provide estimates  on the convergence towards a periodically rotating one-point cluster for \eqref{TCSH} under a priori condition which can be valid at least in a small-time interval. 
 
 \begin{theorem} \label{T4.1} 
 Suppose parameters $\tau, \lambda, \gamma, \varepsilon$ and $\varepsilon_0$ are strictly positive constants satisfying the relations:
\begin{equation} \label{D-0}
0 <\! \tau <\! \infty, \!\quad  0 \!< \varepsilon\le \frac{1}{2}, \!\quad\phi(3\sqrt{2} \varepsilon_0)\!>\!\max\left\{\frac{\varepsilon+\varepsilon\gamma}{\lambda}, \frac{2(T_M+\varepsilon_0)\phi(0)}{(\delta-1)(T_m- \varepsilon_0)}\right\},\!\quad\!\forall~\delta>\!3,
\end{equation}
where \[\lambda =\frac{\kappa_1}{ 2(T_M + \varepsilon_0)},\quad \gamma=\frac{3(\kappa_1\phi(0))^2}{(T_m - \varepsilon_0)^2} +1, \]
and let $\{(\mathbf{x}_{\alpha}, \mathbf{v}_{\alpha}, T_{\alpha}) \}$ be a
    solution to system  \eqref{TCSH} in the time-interval $[0, \tau)$  such that the following initial data condition and a priori condition hold:
    \begin{equation} \label{D-0-0}
   {\mathcal X}(0)\le\varepsilon_0, \quad  \mathcal{V}(0) \le \varepsilon_0,\quad  T_m \ge \delta\varepsilon_0, 
    \quad \sup_{0 \leq t < \tau}  |T_{\alpha}(t) - T^{\infty}(t)| \leq \varepsilon_0,
    \end{equation}
  for some positive constant $\varepsilon_0$. Then, one has an asymptotic convergence toward a periodically rotating one-point cluster for $t \in [0, \tau)$:
 \begin{align}
 \begin{aligned} \label{D-0-1}
& (i)~\sum\limits_{\alpha = 1}^{n} \Big( |\mathbf{x}_\alpha(t) - \mathbf{x}_c(t) |^2  + | \mathbf{v}_\alpha(t) - \mathbf{v}_c(t) |^2 \Big) \\
& \hspace{1cm} \leq 4  \sum\limits_{\alpha = 1}^{n} \Big( |\mathbf{x}_\alpha(0) - \mathbf{x}_c(0) |^2  + | \mathbf{v}_\alpha(0) - \mathbf{v}_c(0) |^2 \Big) e^{-\frac{2\varepsilon}{3} t}. \\
& (ii)~\sum\limits_{\alpha = 1}^{n} |T_\alpha(t) \!-\!T^{\infty}(t)|^2 \leq \sum\limits_{\alpha = 1}^{n} |T_\alpha(0) \!-\!T^{\infty}(0)|^2  e^{-A_1 t}+\frac{A_2}{A_1\!-\!\frac{2}{3}\varepsilon} \Big(e^{-\frac{2\varepsilon}{3} t}\!-\!e^{-A_1 t} \Big).
 \end{aligned}
 \end{align}
 Here $A_1$ and $A_2$ are defined in \eqref{A1A2}.
\end{theorem}
\begin{proof}
 Since the proof is very lengthy, we split its proof into the following two subsections.
\end{proof}
 
 \subsection{Derivation of mechanical flocking} \label{sec:4.1}  For notational simplicity, we set 
\[ |{\mathcal Z}(t)|^2 :=  |\mathcal{X}(t)|^2 + |\mathcal{V}(t)|^2,  \qquad |{\mathbf z}_c(0)|^2 :=  |\mathbf{v}_c(0)|^2+|\mathbf{x}_c(0)|^2, \quad t \geq 0,  \]
where the functionals ${\mathcal X}, {\mathcal V}, {\mathbf x}_c$ and ${\mathbf v}_c$ are defined in \eqref{C-1} and \eqref{A-0}, respectively. 
\begin{proposition}\label{L4.1}
Suppose the following conditions hold.
\begin{enumerate}
\item
Parameters $\lambda$, $\gamma$, $\varepsilon$, $\varepsilon_0$ and $\forall\delta>3$ are strictly positive constants such that 
\begin{equation} \label{D-1-0}
 0 < \varepsilon\le \frac{1}{2}, \quad \phi(3\sqrt{2} \varepsilon_0)>\max\left\{\frac{\varepsilon+\varepsilon\gamma}{\lambda}, \frac{\kappa_1\phi(0)}{(\delta-1)(T_m- \varepsilon_0)\lambda}\right\},\quad T_m \ge \delta\varepsilon_0.
\end{equation}

\vspace{0.1cm}

\item
Mechanical fluctuations $\{ (\hat{\mathbf{x}}_\alpha, \hat{\mathbf{v}}_\alpha)\}$ satisfy the following differential inequality for $\forall~t\ge 0$:
\begin{equation} \label{D-1}
\begin{cases}
\displaystyle \frac{d}{dt}{\mathcal L}(\hat{\mathbf x}, \hat{\mathbf v})
 \le \Big(-2\lambda \phi(\sqrt{2}\mathcal{X})+\varepsilon\gamma+\frac{\kappa_1\varepsilon_0 \phi(0)}{ (T_m - \varepsilon_0)^2 } \Big) \sum\limits_{\alpha = 1}^{n} | \hat{\mathbf{v}}_\alpha|^2-\frac{\varepsilon}{2}\sum\limits_{\alpha = 1}^{n} | \hat{\mathbf{x}}_\alpha|^2, \\
\displaystyle \mathcal{X}(0)\le\varepsilon_0, \quad\mathcal{V}(0)\le\varepsilon_0, 
\end{cases}
\end{equation}
where ${\mathcal L}$ is the functional introduced in \eqref{L}.
\end{enumerate}

\vspace{0.2cm}

Then, one has exponential decay of mechanical fluctuation:
\begin{equation*} \label{D-2}   
  |{\mathcal Z}(t)|^2 \le4  |{\mathcal Z}(0)|^2 e^{-\frac{2\varepsilon}{3} t}, \quad \forall~t\ge 0.
\end{equation*}
 \end{proposition}
 \begin{proof}
 For the proof, we will use a bootstrapping argument.  Note that we cannot use Gr\"onwall's type Lemma \ref{L2.2} for \eqref{D-1} directly, because the coefficient $-2\lambda \phi(\sqrt{2}\mathcal{X})+\varepsilon\gamma+\frac{\kappa_1\varepsilon_0 \phi(0)}{ (T_m - \varepsilon_0)^2 }$ involves with ${\mathcal X}$ which can be  calculated from the solution itself. So we first derive a rough bound estimate for ${\mathcal X}$, and then using this rough bound, we convert the differential inequality \eqref{D-1} into a standard Gr\"onwall's inequality, which yields an exponential decay estimate. \newline

\noindent $\bullet$~Step A (A rough bound estimate for ${\mathcal X}$):  We first show 
    \begin{equation} \label{ubd}
     \sup_{0 \leq t < \infty} \mathcal{X}(t) \leq 3 \varepsilon_0. 
        \end{equation}
 For this, we define a set ${\mathcal S}$: 
 \[ {\mathcal S}:=\{t>0: \mathcal{X}(s) <3\varepsilon_0\quad\mbox{for $s \in(0,t)$}\}.  \]
Since $\mathcal{X}(0)\le\varepsilon_0$, there exists a positive constant $\delta$ such that
\[  \mathcal{X}(s) <3\varepsilon_0, \quad \mbox{for}~s \in [0, \delta). \]
Hence $\delta \in {\mathcal S}$. Now, we set 
\[ \tau^*:=\sup {\mathcal S}. \]
Once we can show $\tau^* = \infty$, we are done. Suppose not, then we have,
\begin{equation} \label{D-4}
\mathcal{X}(t)< 3\varepsilon_0~\mbox{ for $t\in(0, \tau^*)$} \quad \mbox{and} \quad \mathcal{X}(\tau^*)=3\varepsilon_0.
\end{equation}
By the choice of  $\varepsilon\le 1/2$, the functional ${\mathcal L}(\hat{\mathbf x}, \hat{\mathbf v}):=\sum\limits_{\alpha = 1}^{n} (\frac{1}{2} |\hat{\mathbf{x}}_\alpha|^2+ \frac{1}{2}  | \hat{\mathbf{v}}_\alpha|^2 + \varepsilon \hat{\mathbf{x}}_\alpha\cdot  \hat{\mathbf{v}}_\alpha)$ is equivalent to the square of the standard $\ell^2$-norms $ \sum\limits_{\alpha = 1}^{n} (| \hat{\mathbf{x}}_\alpha|^2 + | \hat{\mathbf{v}}_\alpha|^2)$. More precisely, one has 
    \begin{equation} \label{D-4-0}
\frac{3}{16}\sum\limits_{\alpha = 1}^{n}(  | \hat{\mathbf{x}}_\alpha|^2 +  | \hat{\mathbf{v}}_\alpha|^2)\le {\mathcal L}(\hat{\mathbf x}, \hat{\mathbf v})\le \frac{3}{4}\sum\limits_{\alpha = 1}^{n} ( | \hat{\mathbf{x}}_\alpha|^2 + | \hat{\mathbf{v}}_\alpha|^2).
    \end{equation}
 On the other hand, for $t\in(0, \tau^*)$, it follows from \eqref{D-1-0} that
\begin{equation*} 
 -\lambda \phi(3\sqrt{2}\varepsilon_0)+\varepsilon\gamma<-\varepsilon,   \quad-\lambda \phi(3\sqrt{2}\varepsilon_0)+\frac{\kappa_1\varepsilon_0 \phi(0)}{ (T_m - \varepsilon_0)^2} \le 0.
 \end{equation*} 
These and \eqref{D-1} imply
 \begin{align*}
 \begin{aligned}
&\frac{d}{dt}{\mathcal L}(\hat{\mathbf x}, \hat{\mathbf v})
 \le \Big(-\lambda \phi(\sqrt{2}\mathcal{X})+\varepsilon\gamma\Big) \sum\limits_{\alpha = 1}^{n} | \hat{\mathbf{v}}_\alpha|^2-\frac{ \varepsilon}{2}\sum\limits_{\alpha = 1}^{n} | \hat{\mathbf{x}}_\alpha|^2 \le-\frac{ 2\varepsilon}{3}{\mathcal L}(\hat{\mathbf x}, \hat{\mathbf v}).
\end{aligned}
 \end{align*}
 One can  apply Lemma \ref{L2.2} to have
\begin{equation} \label{D-4-1}
 {\mathcal L}(\hat{\mathbf x}, \hat{\mathbf v})(t)
\leq e^{-\frac{2\varepsilon}{3} t}{\mathcal L}(\hat{\mathbf x}, \hat{\mathbf v})(0).
\end{equation}
We combine \eqref{D-4-0} and \eqref{D-4-1} to get 

\begin{equation} \label{D-5-00}
\mathcal{X}^2(t) +\mathcal{V}^2(t) \leq 4 e^{-\frac{2\varepsilon}{3} t}  \Big( |\mathcal{X}(0)|^2  + |\mathcal{V}(0)|^2 \Big)\quad \mbox{for $t\in(0,\tau^*)$}.
\end{equation}
It follows from $\eqref{D-1}_2$ that
\begin{equation*} \label{D-5}
\mathcal{X}(t)\le2\sqrt{2}\varepsilon_0e^{-\frac{\varepsilon}{3}t}\quad \mbox{for $t\in(0,\tau^*)$}.
\end{equation*}
In particular, one has 
\[ {\mathcal X}(\tau^*) \leq 2\sqrt{2}\varepsilon_0 < 3 \varepsilon_0, \]
which is contradictory to \eqref{D-4}. Hence $\tau^* = \infty$.
  \newline
  
\noindent $\bullet$~Step B (Exponential decay estimate for  mechanical fluctuation): ~We use a rough upper bound \eqref{ubd} and repeat the same argument in Step A to get \eqref{D-5-00} for all $t$. 
 
\end{proof}
\begin{remark} \label{rk4.1}
\noindent 1. Note that the differential inequality \eqref{D-1} has been derived in Proposition \ref{P3.1} under the a priori condition \eqref{A-2}. Thus, if we can guarantee that the a priori condition 
\eqref{A-2} is valid in a whole time interval $[0, \infty)$, we will have a desired estimate on the asymptotic formation of a  periodically rotating one-point cluster. \newline

\noindent 2. If we set
\[\lambda =\frac{\kappa_1}{ 2(T_M + \varepsilon_0)},\quad \gamma=\frac{3(\kappa_1\phi(0))^2}{(T_m - \varepsilon_0)^2} +1,\]
then Proposition \ref{L4.1} implies 
\begin{equation} \label{D-5-0}
\mathcal{X}(t)<3\varepsilon_0,\quad   |{\mathcal Z}(t)|^2 \le 4|{\mathcal Z}(0)|^2 e^{-\frac{2\varepsilon}{3} t}\quad \mbox{for $t\in(0, \tau)$}, 
\end{equation}
where we used Lemma \ref{L2.1}, Lemma \ref{L2.2} and Proposition \ref{P3.1} under the assumption \eqref{A-2}.
\end{remark}

\vspace{0.5cm}

 \subsection{Derivation of temperature homogenization} \label{sec:4.2} In this subsection, we study temperature homogenization. As in the mechanical flocking, we first derive a Gr\"onwall's inequality for ${\mathcal T}^2$ and then using Gr\"onwall's type lemma,  we obtain desired temperature estimate. \newline

Note that the equation for ${\hat T}_{\alpha}$ in \eqref{B-5}  can be rewritten as follows:
\begin{align*}
\begin{aligned}
\frac{d{\hat T}_\alpha}{dt} =& -\frac{\kappa_1}{n}\sum_{\beta=1}^{n} \phi(|\hat{\mathbf{x}}_\alpha - \hat{\mathbf{x}}_\beta|)\hat{\mathbf{v}}_\alpha \cdot  \Big( \frac{\hat{\mathbf{v}}_\beta}{{\hat T}_\beta + T^{\infty}} -  \frac{\hat{\mathbf{v}}_\alpha }{{\hat T}_\alpha + T^{\infty}} \Big) \\
&+ \frac{\kappa_2}{n} \sum_{\beta=1}^{n} \zeta(|\hat{\mathbf{x}}_\alpha - \hat{\mathbf{x}}_\beta|)\Big( \frac{1}{{\hat T}_\alpha + T^{\infty}} - \frac{1}{{\hat T}_\beta + T^{\infty}}   \Big)  +\mathbf{\hat{x}}_\alpha \cdot\mathbf{\hat{v}}_\alpha+\mathbf{v}_c\cdot\mathbf{\hat{x}}_\alpha+\mathbf{x}_c\cdot\mathbf{\hat{v}}_\alpha.
\end{aligned}
\end{align*}
We multiply the above equation by ${\hat T}_{\alpha}$ and then sum the  resulting relation over $\alpha$ to obtain
\begin{align}
\begin{aligned} \label{D-6}
\frac{d}{dt} \sum\limits_{\alpha=1}^{n} \frac{{\hat T}^2_\alpha}{2} = &-\frac{\kappa_1}{n} \sum\limits_{\alpha, \beta=1}^{n} \phi(|\hat{\mathbf{x}}_\alpha - \hat{\mathbf{x}}_\beta|) {\hat T}_{\alpha} \hat{\mathbf{v}}_\alpha \cdot \Big( \frac{\hat{\mathbf{v}}_\beta}{{\hat T}_\beta + T^{\infty}} -
    \frac{\hat{\mathbf{v}}_\alpha}{{\hat T}_\alpha + T^{\infty}} \Big) \\
&+\frac{\kappa_2}{n} \sum\limits_{\alpha, \beta=1}^{n} \zeta(|\hat{\mathbf{x}}_\alpha - \hat{\mathbf{x}}_\beta|){\hat T}_{\alpha} \Big( \frac{1}{{\hat T}_\alpha + T^{\infty}} - \frac{1}{{\hat T}_\beta + T^{\infty}}  \Big)  \\
& + \sum\limits_{\alpha=1}^n {\hat T}_\alpha\mathbf{\hat{x}}_\alpha \cdot\mathbf{\hat{v}}_\alpha+\mathbf{v}_c\cdot\sum\limits_{\alpha=1}^n {\hat T}_\alpha\mathbf{\hat{x}}_\alpha+\mathbf{x}_c\cdot\sum\limits_{\alpha=1}^n {\hat T}_\alpha\mathbf{\hat{v}}_\alpha \\
=:& {\mathcal I}_{21} + {\mathcal I}_{22} + {\mathcal I}_{23} .
\end{aligned}
\end{align}

Now we derive a  Gr\"onwall's type inequality for $ \mathcal T^2$.
\begin{proposition}\label{P4.2a}
Suppose the  conditions  \eqref{D-0} and \eqref{D-0-0} hold, and let
$\{ (\hat{\mathbf{x}}_{\alpha}, \hat{\mathbf{v}}_{\alpha}, \hat{T}_{\alpha}) \}$ be a solution to system \eqref{B-5} in the time interval $[0,\tau)$. 
           Then  
           we  have
  \begin{align*}
    \begin{aligned}
\frac{d{\mathcal T}^2}{dt} &\leq~2\Big(-\frac{\kappa_2\zeta(3\sqrt{2} \varepsilon_0)}{(T_M+ \varepsilon_0)^2}+\frac{\kappa_1\phi(0)}{2(T_m- \varepsilon_0)}+\frac{1}{2} +\frac{8\kappa_1\phi(0)}{(T_m- \varepsilon_0)^2} |{\mathcal Z}(0)|^2 e^{-\frac{2\varepsilon}{3} t}+ \sqrt{2}  |{\mathbf z}_c(0)| \Big) {\mathcal T}^2\\
&\ \ \ \ \ +~2\frac{\phi(0)\kappa_1}{T_m- \varepsilon_0} {\mathcal V}^4+\frac{\kappa_2\zeta(3\sqrt{2} \varepsilon_0)}{2n(T_M+ \varepsilon_0)^2} {\mathcal V}^4+ {\mathcal X}^2 {\mathcal V}^2+  \sqrt{2}  |{\mathbf z}_c(0)| \cdot |{\mathcal Z}|^2, \quad t\in[0,\tau).
  \end{aligned}
    \end{align*}
\end{proposition}
\begin{proof} In \eqref{D-6}, we provide estimates for ${\mathcal I}_{2i}$ separately. 

\vspace{0.5cm}

\noindent $\bullet$ (Estimate for ${\mathcal I}_{21}$):~We use \eqref{C-4} to obtain
\begin{align}
\begin{aligned} \label{D-8}
 {\mathcal I}_{21} = &-\frac{\kappa_1}{n}\sum\limits_{\alpha, \beta=1}^{n} \phi(|\hat{\mathbf{x}}_\alpha - \hat{\mathbf{x}}_\beta|) {\hat T}_{\alpha} \hat{\mathbf{v}}_\alpha \cdot  \Big( \frac{\hat{\mathbf{v}}_\beta}{{\hat T}_\beta + T^{\infty}} -
    \frac{\hat{\mathbf{v}}_\alpha}{{\hat T}_\alpha + T^{\infty}} \Big) \\
    =&  -\frac{\kappa_1}{n} \sum\limits_{\alpha, \beta=1}^{n} \phi(|\hat{\mathbf{x}}_\alpha - \hat{\mathbf{x}}_\beta|)\frac{{\hat T}_{\alpha}}{{\hat T}_{\beta} + T^{\infty}}  \hat{\mathbf{v}}_\alpha \cdot (\hat{\mathbf{v}}_{\beta} - \hat{\mathbf{v}}_{\alpha})  \\
    &- \frac{\kappa_1}{n} \sum\limits_{\alpha, \beta=1}^{n} \phi(|\hat{\mathbf{x}}_\alpha - \hat{\mathbf{x}}_\beta|)
\frac{|\hat{\mathbf{v}}_{\alpha}|^2}{ ({\hat T}_{\alpha} + T^{\infty}) ({\hat T}_{\beta} + T^{\infty})  }  {\hat T}_{\alpha}
    ({\hat T}_{\alpha} - {\hat T}_{\beta}) \\
    =:&~{\mathcal I}_{211} + {\mathcal I}_{212}.
\end{aligned}
\end{align}
\noindent $\diamond$  (Estimate for ${\mathcal I}_{211}$):~By direct calculation, one has
\begin{align}
\begin{aligned} \label{D-9}
    {\mathcal I}_{211}  &\le \frac{1}{2n}\frac{\kappa_1\phi(0)}{(T_{m} - \varepsilon_0)}\sum\limits_{\alpha, \beta=1}^{n} ({\hat T}^2_{\alpha}+ |\hat{\mathbf{v}}_{\alpha}|^2
|\hat{\mathbf{v}}_{\beta} - \hat{\mathbf{v}}_{\alpha}|^2)\\
&\le \frac{1}{2}\frac{\kappa_1\phi(0)}{(T_m - \varepsilon_0)} {\mathcal T}^2+\frac{\kappa_1\phi(0)}{(T_m - \varepsilon_0)} {\mathcal V}^4.
 \end{aligned}
 \end{align}

\noindent $\diamond$ (Estimate for ${\mathcal I}_{212}$): ~We use the  short-time estimate \eqref{D-5-0} 
 to find
\begin{align}
\begin{aligned} \label{D-10}
    {\mathcal I}_{212} &:= -\frac{\kappa_1}{n} \sum\limits_{\alpha, \beta=1}^{n} \phi(|\hat{\mathbf{x}}_\alpha - \hat{\mathbf{x}}_\beta|)
\frac{|\hat{\mathbf{v}}_{\alpha}|^2}{ ({\hat T}_{\alpha} + T^{\infty}) ({\hat T}_{\beta} + T^{\infty})  }   {\hat T}_{\alpha}
    ({\hat T}_{\alpha} - {\hat T}_{\beta}) \cr
    &\leq \!4\frac{\phi(0)\kappa_1}{n(T^{\infty}\!-\!\varepsilon_0)^2}   |{\mathcal Z}(0)|^2  e^{-\frac{2\varepsilon}{3} t}\!\!\!\sum_{\alpha, \beta=1}^{n}\!({\hat T}_{\alpha}^2+{\hat T}_{\alpha}{\hat T}_{\beta}) \leq \frac{8\phi(0)\kappa_1 |{\mathcal Z}(0)|^2 }{(T_m\!-\! \varepsilon_0)^2}  e^{-\frac{2\varepsilon}{3} t}{\mathcal T}^2.
    \end{aligned}
     \end{align}
\noindent $\bullet$  (Estimate for ${\mathcal I}_{22}$): ~We use the relation $\mathcal{X}(t)< 3 \varepsilon_0$ in \eqref{D-5-0} to see
\begin{align}
\begin{aligned} \label{D-11}
    {\mathcal I}_{22} =&  \frac{\kappa_2}{n}\sum\limits_{\alpha, \beta=1}^{n} \zeta(|\hat{\mathbf{x}}_\alpha - \hat{\mathbf{x}}_\beta|) {\hat T}_{\alpha} \Big( \frac{1}{{\hat T}_\alpha + T^{\infty}} - \frac{1}{{\hat T}_\beta + T^{\infty}}   \Big) \cr
    =& -\frac{\kappa_2}{n}\sum\limits_{\alpha, \beta=1}^{n} \zeta(|\hat{\mathbf{x}}_\alpha - \hat{\mathbf{x}}_\beta|) {\hat T}_{\beta} \Big( \frac{1}{{\hat T}_\alpha + T^{\infty}} - \frac{1}{{\hat T}_\beta + T^{\infty}}   \Big) \cr
    =& \frac{\kappa_2}{2n} \sum\limits_{\alpha, \beta=1}^{n} \zeta(|\hat{\mathbf{x}}_\alpha - \hat{\mathbf{x}}_\beta|) ({\hat T}_{\alpha} - {\hat T}_{\beta}) \Big( \frac{1}{{\hat T}_\alpha + T^{\infty}} - \frac{1}{{\hat T}_\beta + T^{\infty}}   \Big) \cr
    =&  -\frac{\kappa_2}{2n} \sum\limits_{\alpha, \beta=1}^{n} \zeta(|\hat{\mathbf{x}}_\alpha - \hat{\mathbf{x}}_\beta|) \frac{
        ({\hat T}_{\alpha} - {\hat T}_{\beta})^2}{({\hat T}_\alpha + T^{\infty})({\hat T}_\beta + T^{\infty})} \cr
    \leq& -\frac{\kappa_2\zeta(\sqrt{2}{\mathcal X})}{2n(T_{M}+ \varepsilon_0)^2}  \sum\limits_{\alpha, \beta=1}^{n} ({\hat T}_{\alpha} - {\hat T}_{\beta})^2\cr
     \leq& -\frac{\kappa_2\zeta(3\sqrt{2}\varepsilon_0)}{(T_M+ \varepsilon_0)^2} {\mathcal T}^2+\frac{\kappa_2\zeta(3\sqrt{2}\varepsilon_0)}{4n(T_M+ \varepsilon_0)^2} {\mathcal V}^4,
\end{aligned}
\end{align}
where we used the relation:
\begin{align}
\begin{aligned} \label{D-11-0}
\sum\limits_{\alpha,\beta=1}^n({\hat T}_\alpha-{\hat T}_\beta)^2 &= \sum\limits_{\alpha,\beta=1}^n| (T_\alpha- T^{\infty})-(T_\beta-T^{\infty})|^2 \cr
&= 2n  \sum\limits_{\alpha=1}^n |T_\alpha - T^{\infty}|^2  - 2 \Big( \sum\limits_{\alpha=1}^n (T_\alpha - T^{\infty}) \Big)^2 = 2n {\mathcal T}^2-\frac{1}{2} {\mathcal V}^4. 
\end{aligned}
\end{align}
Note that we used the energy conservation law given by Lemma \ref{L2.1} in \eqref{D-11-0}:
\[\sum\limits_{\alpha=1}^n (T_\alpha + \frac{1}{2} |\mathbf{v}_\alpha|^2) = \sum\limits_{\alpha=1}^n (T_\alpha(0) + \frac{1}{2} |\mathbf{v}_\alpha(0)|^2) = n T^{\infty}+\frac{n}{2}|\mathbf{v}_c|^2,\]
which is equivalent to
\[\sum\limits_{\alpha=1}^n (T_\alpha - T^{\infty}) = -\frac{1}{2} \sum\limits_{\alpha=1}^n (|\mathbf{v}_\alpha|^2 -|\mathbf{v}_c|^2)= -\frac{1}{2} \sum\limits_{\alpha=1}^n (\hat{\mathbf{v}}_\alpha+2\mathbf{v}_c)\cdot\hat{\mathbf{v}}_\alpha= -\frac{1}{2} {\mathcal V}^2.\]
\noindent $\bullet$  (Estimate for ${\mathcal I}_{23}$): ~We use \eqref{A-0} to get
\begin{align}
\begin{aligned} \label{D-12}
{\mathcal I}_{23}=&\sum\limits_{\alpha=1}^n {\hat T}_\alpha\mathbf{\hat{x}}_\alpha\cdot\mathbf{\hat{v}}_\alpha+\mathbf{v}_c\cdot\sum\limits_{\alpha=1}^n {\hat T}_\alpha\mathbf{\hat{x}}_\alpha+\mathbf{x}_c\cdot\sum\limits_{\alpha=1}^n {\hat T}_\alpha\mathbf{\hat{v}}_\alpha\\
\leq& \frac{1}{2} {\mathcal T}^2+ \frac{1}{2}\sum\limits_{\alpha=1}^n|\hat{\mathbf{x}}_\alpha|^2|\hat{\mathbf{v}}_\alpha|^2+|\mathbf{v}_c| \Big( \frac{1}{2} {\mathcal T}^2+ \frac{1}{2} {\mathcal X}^2 \Big)+|\mathbf{x}_c| \Big( \frac{1}{2} {\mathcal T}^2+ \frac{1}{2} {\mathcal V}^2 \Big)\\
 \leq& \frac{1}{2} {\mathcal T}^2+ \frac{1}{2} {\mathcal X}^2 {\mathcal V}^2 +  \sqrt{2}  |{\mathbf z}_c(0)|  \Big({\mathcal T}^2+\frac{1}{2}{\mathcal V}^2+\frac{1}{2} {\mathcal X}^2 \Big),
\end{aligned}
\end{align}
since $|\mathbf{x}_c(t)|^2\le 2|{\mathbf z}_c(0)|^2$ and $|\mathbf{v}_c(t)|^2\le 2|{\mathbf z}_c(0)|^2$.\newline

In \eqref{D-6}, we collect all the estimates \eqref{D-8}, \eqref{D-9}, \eqref{D-10}, \eqref{D-11} and \eqref{D-12} to obtain
\begin{align*}
\frac{d{\mathcal T}^2}{dt} \leq& 2\Big(-\frac{\kappa_2\zeta(3\sqrt{2} \varepsilon_0)}{(T_M+ \varepsilon_0)^2}+\frac{\kappa_1\phi(0)}{2(T_m- \varepsilon_0)}+\frac{1}{2} +\frac{8\kappa_1\phi(0)}{(T_m- \varepsilon_0)^2} |{\mathcal Z}(0)|^2 e^{-\frac{2\varepsilon}{3} t}+ \sqrt{2}  |{\mathbf z}_c(0)| \Big) {\mathcal T}^2\\
&+2\frac{\phi(0)\kappa_1}{T_m- \varepsilon_0} {\mathcal V}^4+\frac{\kappa_2\zeta(3\sqrt{2} \varepsilon_0)}{2n(T_M+ \varepsilon_0)^2} {\mathcal V}^4+ {\mathcal X}^2 {\mathcal V}^2+  \sqrt{2}  |{\mathbf z}_c(0)| \cdot |{\mathcal Z}|^2.
 \end{align*}
\end{proof}
For a given $\varepsilon_0 > 0$, we set $A_i,~i=1,2$ as follows.
\begin{align}
\begin{aligned} \label{A1A2}
A_1 &:=\frac{2\kappa_2\zeta(3\sqrt{2} \varepsilon_0)}{(T_M+ \varepsilon_0)^2}-\frac{\kappa_1\phi(0)}{T_m- \varepsilon_0}-1 -\frac{16\kappa_1\phi(0)}{(T_m- \varepsilon_0)^2} |{\mathcal Z}(0)|^2 - 2 \sqrt{2}  |{\mathbf z}_c(0)|  > 0, \\
A_2 &:= 16\left(\frac{2\phi(0)\kappa_1}{T_m- \varepsilon_0}+\frac{\kappa_2\zeta(3\sqrt{2} \varepsilon_0)}{2n(T_M+ \varepsilon_0)^2}+1\right) |{\mathcal Z}(0)|^4 + 
 4 \sqrt{2}  |{\mathbf z}_c(0)|  |{\mathcal Z}(0)|^2> 0.
\end{aligned}
\end{align}
Note that whenever $\kappa_2$ becomes larger, $A_2$ also becomes larger, but $\frac{A_2}{|A_1-\frac{2}{3}\varepsilon|}$ is bounded in the following context. So from now on we assume that $A_1>0$.

Now we show that the temperature fluctuation functional ${\mathcal T}^2$ decays  exponentially at least in a short-time interval.
\begin{proposition}\label{P4.2}
Suppose the  conditions  \eqref{D-0} and \eqref{D-0-0} hold, let
    $\{ (\hat{\mathbf{x}}_{\alpha}, \hat{\mathbf{v}}_{\alpha}, \hat{T}_{\alpha}) \}$ be a
    solution to system \eqref{B-5} for  $\tau \in (0, \infty]$. Then, we have temperature homogenization: 
    \begin{equation*} \label{D-14}
        \mathcal{T}^2(t)\le \mathcal{T}^2(0)e^{-A_1 t}+\frac{A_2}{A_1-\frac{2}{3}\varepsilon} \Big(e^{-\frac{2\varepsilon}{3} t}-e^{-A_1 t} \Big), \quad t\in(0, \tau),
    \end{equation*}
    where $\varepsilon$ satisfies \eqref{D-0}.
\end{proposition}
\begin{proof}
 It follows from Proposition \ref{P4.2a}  and Remark \ref{rk4.1}  that 
\begin{align}
\begin{aligned}\label{D-15}
     &\frac{d{\mathcal T}^2}{dt}\!\! \leq 2\Big(\frac{-\kappa_2\zeta(3\sqrt{2} \varepsilon_0)}{(T_M+ \varepsilon_0)^2}\!+\!\frac{\kappa_1\phi(0)}{2(T_m- \varepsilon_0)}\!+\!\frac{1}{2} \!+\!\frac{8\phi(0)\kappa_1|{\mathcal Z}(0)|^2}{(T_m- \varepsilon_0)^2}   e^{-\frac{2\varepsilon}{3} t}\!+\! \sqrt{2}  |{\mathbf z}_c(0)| \Big) {\mathcal T}^2\\
&\ \ \ \ + \Big( 2\frac{\phi(0)\kappa_1}{T_m- \varepsilon_0} +\frac{\kappa_2\zeta(3\sqrt{2} \varepsilon_0)}{2n(T_M+ \varepsilon_0)^2} \Big) {\mathcal V}^4+ {\mathcal X}^2 {\mathcal V}^2+ \sqrt{2}  |{\mathbf z}_c(0)|  |{\mathcal Z}|^2 \\
   &\le-\left(\frac{2\kappa_2\zeta(3\sqrt{2} \varepsilon_0)}{(T_M+ \varepsilon_0)^2}-\frac{\kappa_1\phi(0)}{T_m- \varepsilon_0}-1-\frac{16\phi(0)\kappa_1}{(T_m- \varepsilon_0)^2}  |{\mathcal Z}(0)|^2 - 2\sqrt{2 } |{\mathbf z}_c(0)| \right) {\mathcal T}^2\\
   &\ \ \ \ + \!16\! \left(\!\frac{2\phi(0)\kappa_1}{T_m- \varepsilon_0}\!+\frac{\kappa_2\zeta(3\sqrt{2} \varepsilon_0)}{2n(T_M+ \varepsilon_0)^2}\!+\!1\!\right)\!\! |{\mathcal Z}(0)|^4 e^{-\frac{4\varepsilon}{3} t}\!+\!4\sqrt{2}  |{\mathbf z}_c(0)| |{\mathcal Z}(0)|^2 e^{-\frac{2\varepsilon}{3} t}\\
   &\leq -\left(\frac{2\kappa_2\zeta(3\sqrt{2} \varepsilon_0)}{(T_M+ \varepsilon_0)^2}-\frac{\kappa_1\phi(0)}{T_m- \varepsilon_0}-1-\frac{16\phi(0)\kappa_1}{(T_m- \varepsilon_0)^2}  |{\mathcal Z}(0)|^2 - 2\sqrt{2}  |{\mathbf z}_c(0)| \right) {\mathcal T}^2\\
  &\ \ \ \ +\!\left[ 16\!\left(\frac{2\phi(0)\kappa_1}{T_m- \varepsilon_0}\!+\!\frac{\kappa_2\zeta(3\sqrt{2} \varepsilon_0)}{2n(T_M+ \varepsilon_0)^2}\!+\!1\!\right) \!\!|{\mathcal Z}(0)|^4
 +4\sqrt{2}  |{\mathbf z}_c(0)|  |{\mathcal Z}(0)|^2 \right] \!e^{-\frac{2\varepsilon}{3} t} \\
  &=  -A_1 {\mathcal T}^2+A_2e^{-\frac{2\varepsilon}{3} t}.
  \end{aligned}
    \end{align}
 Finally, we apply Lemma \ref{L2.2}  for \eqref{D-15} to see
  \[ \mathcal{T}^2(t)\le \mathcal{T}^2(0)e^{-A_1 t}+\frac{A_2}{A_1-\frac{2}{3} \varepsilon }(e^{-\frac{2\varepsilon}{3} t}-e^{-A_1 t}),\quad \mbox{for $0\le t< \tau.$}\]
\end{proof}
\begin{remark}
Note that if we assume a rough bound estimate of temperature fluctuations around time-dependent temperature $T^{\infty}(t)$, then we can obtain exponential decay estimate toward $T^{\infty}(t)$. This bootstrapping argument will be used to derive the thermo-mechanical flocking in the whole time interval $(0, \infty)$ in next section. 
\end{remark}

\section{Emergence of a periodically rotating one-point cluster} \label{sec:5}
\setcounter{equation}{0}
In this section, we provide our main result on the formation of a periodically rotating one-point cluster by extending a local result of Theorem \ref{T4.1} in the small-time interval $[0, \tau)$  to the whole line. For this, we verify  that the a priori assumption \eqref{D-0-0} on the positivity of temperatures, i.e., \eqref{A-2} holds in a whole time interval using the continuous induction and decay estimates of fluctuations in Theorem \ref{T4.1}.  We are now ready to present our main result on the emergence to asymptotic flocking. 
\begin{theorem} \label{T5.1}

Suppose that the communication weights and initial data satisfy the following relations:
\begin{enumerate}
\item
The communication weight functions $\Phi$ and $\zeta$ satisfy (\ref{CA}) and (\ref{comm}).
\item
There exists $\varepsilon_0>0$ such that the initial data satisfy 
 \begin{align} 
 \begin{aligned} \label{E-1}
      \quad T_m \ge \delta\varepsilon_0 \quad ( \forall \delta>3) , \quad 
         {\mathcal X}(0)\le\varepsilon_0, \quad  \mathcal{V}(0) \le \varepsilon_0,  \quad
    \varepsilon_0 ^2
> \mathcal{T}^2(0)+\frac{A_2}{|A_1-\frac{2}{3} \varepsilon|}.
\end{aligned}
\end{align}
\item
Moreover, the parameters $\varepsilon$  and  $\varepsilon_0$  in \eqref{E-1} satisfy
 \begin{align*} 
 \begin{aligned}
 &0 <\varepsilon ( \ne \frac{3}{2}A_1)  \le \frac{1}{2}, \quad \phi(3\sqrt{2} \varepsilon_0)>\max\left\{\frac{\varepsilon+\varepsilon\gamma}{\lambda}, \frac{2(T_M+\varepsilon_0)\phi(0)}{(\delta-1)(T_m- \varepsilon_0)}\right\}, \\
&\lambda:= \frac{\kappa_1}{ 2(T_M + \varepsilon_0)}, \quad \gamma := \frac{3(\kappa_1\phi(0))^2}{(T_m - \varepsilon_0)^2} +1,
 \end{aligned}
\end{align*}

\end{enumerate} 
 and let $\{ (\mathbf{x}_{\alpha}, \mathbf{v}_{\alpha}, T_{\alpha}) \}$ be a solution to system  \eqref{TCSH}.
Then, the following assertions hold:
\begin{enumerate}
\item The  temperature fluctuations are uniformly bounded:
\begin{equation} \label{E-2}
\sup_{0 \leq t < \infty} |T_\alpha(t) - T^{\infty}(t) |\le \varepsilon_0,
\end{equation}
 or the temperatures are uniformly bounded and away from zero: 
\begin{equation*} \label{Ap-nm-2}
0\ll T_m -\varepsilon_0\le T_\alpha(t) \le T_M+ \varepsilon_0, \quad t\in[0,\infty), \quad \alpha=1,\cdots n.
\end{equation*}

\vspace{0.2cm}

\item
The thermo-mecanical flocking occurs exponentially fast: for $t \in [0, \infty)$, 
\begin{align}
 \begin{aligned} \label{E-3}
& \sum\limits_{\alpha = 1}^{n} \Big( |\mathbf{x}_\alpha(t) - \mathbf{x}_c(t) |^2  + | \mathbf{v}_\alpha(t) - \mathbf{v}_c(t) |^2 \Big) \\
& \hspace{0.5cm} \leq 4  \sum\limits_{\alpha = 1}^{n} \Big( |\mathbf{x}_\alpha(0) - \mathbf{x}_c(0) |^2  + | \mathbf{v}_\alpha(0) - \mathbf{v}_c(0) |^2 \Big) e^{-\frac{2\varepsilon}{3} t}. \\
& \sum\limits_{\alpha = 1}^{n} |T_\alpha(t) -T^{\infty}(t)|^2 \leq \sum\limits_{\alpha = 1}^{n} |T_\alpha(0) -T^{\infty}(0)|^2  e^{-A_1 t}+\frac{A_2}{A_1-\frac{2}{3}\varepsilon} \Big(e^{-\frac{2\varepsilon}{3} t}-e^{-A_1 t} \Big).
 \end{aligned}
 \end{align}
\end{enumerate}
\end{theorem}
\begin{proof} For the desired estimates \eqref{E-3}, it follows from Theorem \ref{T4.1} that it suffices to check the a priori condition  in \eqref{D-0-0} holds for $\tau = \infty$.  \newline

For this, define
\[\mathcal{S}:= \Big \{t>0:T^{\infty}(s) - \varepsilon_0\le T_\alpha(s)\le T^{\infty}(s) + \varepsilon_0, \quad s \in [0, t), \quad \alpha=1,\cdots,n \Big \}.\]
Then, since 
\[ T_m- \varepsilon_0 \leq T^{\infty}(0) -\varepsilon_0 < T_{\alpha}(0) < T^{\infty}(0) + \varepsilon_0 \leq T_M + \varepsilon_0 \]
and by the continuity, there exists $\tau^{\prime} > 0$ such that 
\[  T^{\infty}(s) - \varepsilon_0\le T_\alpha(s)\le T^{\infty}(s) + \varepsilon_0, \quad s \in [0, \tau^{\prime}).   \]
Hence 
\[ \tau^{\prime} \in {\mathcal S}, \quad \mbox{i.e.,} \quad {\mathcal S} \not = \emptyset. \]
Now we claim:
\[ \tilde{t}^* :=  \sup {\mathcal S} = \infty. \]
Suppose not, i.e., ${\tilde t}^*<\infty$. Then, there exists at least one $\alpha$ such that
\[T_\alpha({\tilde t}^*)=T^{\infty}({\tilde t}^*) + \varepsilon_0 \quad \mbox{or} \quad T_\alpha({\tilde t}^*)=T^{\infty}({\tilde t}^*) - \varepsilon_0.\]
In what follows, we will show that the above two cases lead to contradictions, and we conclude $ \tilde{t}^* = \infty$ and the desired estimates follow from Theorem \ref{T4.1}.

\vspace{0.5cm}

\noindent $\bullet$ (Case A): Suppose there exists $\alpha \in \{1, \cdots, n \}$ such that 
\[ T_\alpha({\tilde t}^*)=T^{\infty}({\tilde t}^*)  + \varepsilon_0. \]
This and $\eqref{C-1}_3$ yield
\begin{equation} \label{E-4}
\varepsilon_0^2 = |T_\alpha({\tilde t}^*)-T^\infty({\tilde t}^*)|^2 \leq \mathcal{T}^2({\tilde t}^*).
\end{equation}
On the other hand, we use Proposition \ref{P4.2} to obtain an upper bound for $\mathcal{T}^2$:
\begin{equation}\label{E-5}
\mathcal{T}^2({\tilde t}^*) \le \mathcal{T}^2(0) e^{-A_1 {\tilde t}^*}+\frac{A_2}{A_1-\frac{2}{3}\varepsilon}(e^{-\frac{2\varepsilon }{3}{\tilde t}^*}-e^{-A_1 {\tilde t}^*}) \le \mathcal{T}^2(0)+ \frac{A_2}{|A_1-\frac{2}{3}\varepsilon|}.
\end{equation}
Now, it follows from \eqref{E-4} and \eqref{E-5} that 
\[ \varepsilon_0 ^2 \le \mathcal{T}^2(0)+\frac{A_2}{|A_1-\frac{2}{3}\varepsilon|}. \]
This is contradictory to the last relation in \eqref{E-1}.  \newline

\noindent $\bullet$ (Case B):~Suppose there exists $\alpha \in \{1, \cdots, n \}$ such that 
\[ T_\alpha({\tilde t}^*)=T^{\infty}({\tilde t}^*)- \varepsilon_0 \]
which leads to the same estimate \eqref{E-4}. Then, we can repeat the same argument in Case A to get a contradiction to \eqref{E-1}. Finally, it follows from Case A and Case B that we derive a contradiction from the hypothesis that ${\tilde t}^*$ is finite. Therefore, we have
\[ \tilde{t}^* = \infty   \]
and the a priori condition on the positivity of temperatures is valid in a whole time interval as in \eqref{E-2}:
\[   
 \sup_{0 \leq t < \infty}  |{\hat T}_{\alpha}(t)| \leq \varepsilon_0. \]
Then the thermo-mechanical estimates \eqref{D-0-1} in Theorem \ref{T4.1} hold for $\tau=\infty$. 
\end{proof}

\begin{remark} \label{R5.5}
   We can relax the conditions (2)-(3)  in Theorem  \ref{T5.1} to more applicable  conditions:
   there exist $\varepsilon>0$ and $\varepsilon_0>0$ such that the initial data satisfy 
   \begin{align*}   
 \begin{aligned}
&(2'): {\mathcal X}(0)\le\varepsilon_0, \quad  \mathcal{V}(0) \le \varepsilon_0, \quad  \varepsilon_0 ^2
> \mathcal{T}^2(0)+\frac{A_2}{|A_1-\frac{2}{3} \varepsilon|},\\
& (3'):  -2\lambda \phi(3\sqrt{2} \varepsilon_0)+\varepsilon\gamma+\frac{\kappa_1\varepsilon_0 \phi(0)}{ (T_m - \varepsilon_0)^2 }\le-\varepsilon.
 \end{aligned}
\end{align*}

If we choose ${\mathcal X}(0)\le\varepsilon_0$, $\mathcal{V}(0) \le \varepsilon_0$, $\mathcal{T}^2(0) \le \frac{1}{2}\varepsilon_0^2$,  $\kappa_2\zeta(3\sqrt{2} \varepsilon_0)\gg\varepsilon_0^2(T_M+ \varepsilon_0)^2\gg \kappa_1$  and  $\ \frac{(T_M-\varepsilon_0)^2}{T_M+\varepsilon_0}>3\sqrt{2} \varepsilon_0^2$, then $(2')$ and  $(3')$   are true. 
 \end{remark}
\begin{remark} \label{R5.1}
In a recent work \cite{ST}, Shu and Tadmor studied emergent dynamics for the hydrodynamic Cucker-Smale model in an external field:
\[
\begin{cases}
\displaystyle \partial_t \rho + \nabla_x (\rho {\mathbf u}) = 0, \\
\displaystyle \partial_t {\mathbf u} + {\mathbf u}  \cdot \nabla_{{\mathbf x}} {\mathbf u}  = \int \phi(|{\mathbf x} - {\mathbf y}|) ({\mathbf u}({\mathbf y}, t) - {\mathbf u}({\mathbf x}, t))\rho({\mathbf y}, t) d{\mathbf y} - \nabla U({\mathbf x}),
\end{cases}
\]
and obtained several emergent estimates using an energy estimate and Lagrangian approach based on particle trajectories.
\end{remark}

%
%
%
\section{Numerical simulation}  \label{sec:6}
\setcounter{equation}{0} 
In this section, we provide numerical examples for the thermo-mechanical  flocking. 
For the simulation of system  \eqref{TCSH}, we use the forth-order Runge-Kutta method with a  time step $\Delta t=0.01$ in the two dimensional Euclidean space.
We set $n=100$ and pick initial data as in Figure \ref{fig1}. Precisely we take $\mathbf x_\alpha \in (0.32,0.35)\times(0.2,0.24)$, $\mathbf v_\alpha\in (-0.3,-0.29)\times(0.05,0.06)$, and $T_\alpha\in (10.8,10.9)$ for $\alpha =1,\ldots, n$ which are picked randomly among the rational numbers on each associated interval.

We choose the coupling strengths and the communication weight functions as follows:
\[ \kappa_1=1,\quad\kappa_2=100,\quad\phi(r)=\frac{1}{\sqrt{1+r^2}},\quad \zeta(r)=\frac{40}{\sqrt{1+r^2}}. \]
 From these, we compute some values:
\[ \mathcal{X}(0)\approx0.1419,\quad\mathcal{V}(0)\approx0.5470,\quad \mathcal{T}(0)\approx0.2722,
\quad T_m\approx 10.6445, \quad T_M\approx 10.8955. \]
\begin{figure}[h!] \centering
\includegraphics[width=0.9\textwidth]{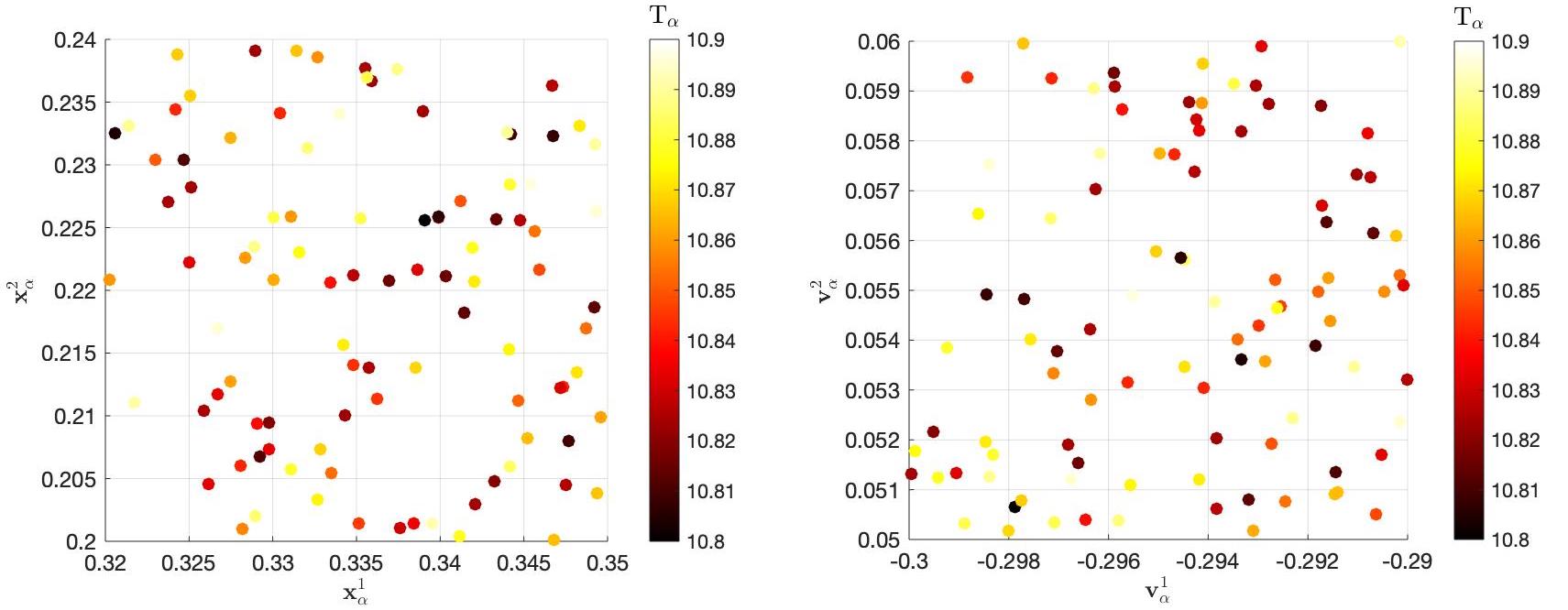}
\caption{Initial positions(left), velocities(right) and temperatures(color) with $n=100$.} \label{fig1}
\end{figure}

\noindent Now with the  setting $\varepsilon_0=0.76,\,\varepsilon=0.003$, 
these parameters  satisfy the conditions  $(2)$ and $(3)$ in Theorem \ref{T5.1}. 
In Figure \ref{fig3}, we can see the flocking phenomena of $\mathbf{x}_\alpha,\,\mathbf{v}_\alpha$ and $\mathbf{T}_\alpha$. Indeed these configuration which appear to be periodic, approximate to $\mathbf{x}_c$,  $\mathbf{v}_c$, and $T^\infty$ respectively which are periodic (Lemma \ref{L2.1}).
In Figure \ref{fig5}, we can see the exponential decay of $\ell^2$ fluctuations $\mathcal{X}$, $ \mathcal{V}$ and $\mathcal T$. From the decay results, we can assert that $\mathbf{x}_\alpha,\,\mathbf{v}_\alpha$ and $\mathbf{T}_\alpha$ approximate $\mathbf{x}_c$, $\mathbf{v}_c$, and $\mathbf{T}^\infty$ respectively.
\begin{figure}[h!] \centering
\includegraphics[width=1.0\textwidth]{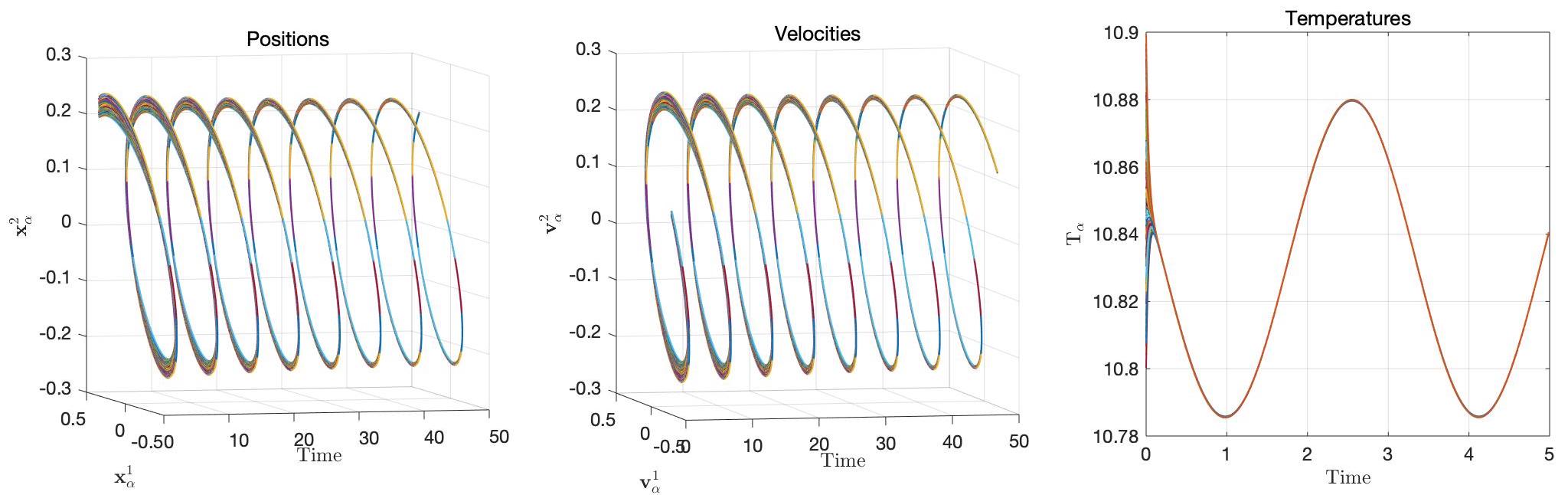}\caption{Dynamics of positions $\mathbf x_\alpha$(left), velocities $\mathbf v_\alpha$(mid), and temperatures $T_\alpha$(right). Each line with a color shows one particle. $\mathbf{x}_\alpha^i,\,\mathbf{v}_\alpha^i$ mean $i$-th arguments of $\mathbf{x}_\alpha,\,\mathbf{v}_\alpha$ respectively.}\label{fig3}
\end{figure}
\begin{figure}[h!] \centering
\includegraphics[width=1.0\textwidth]{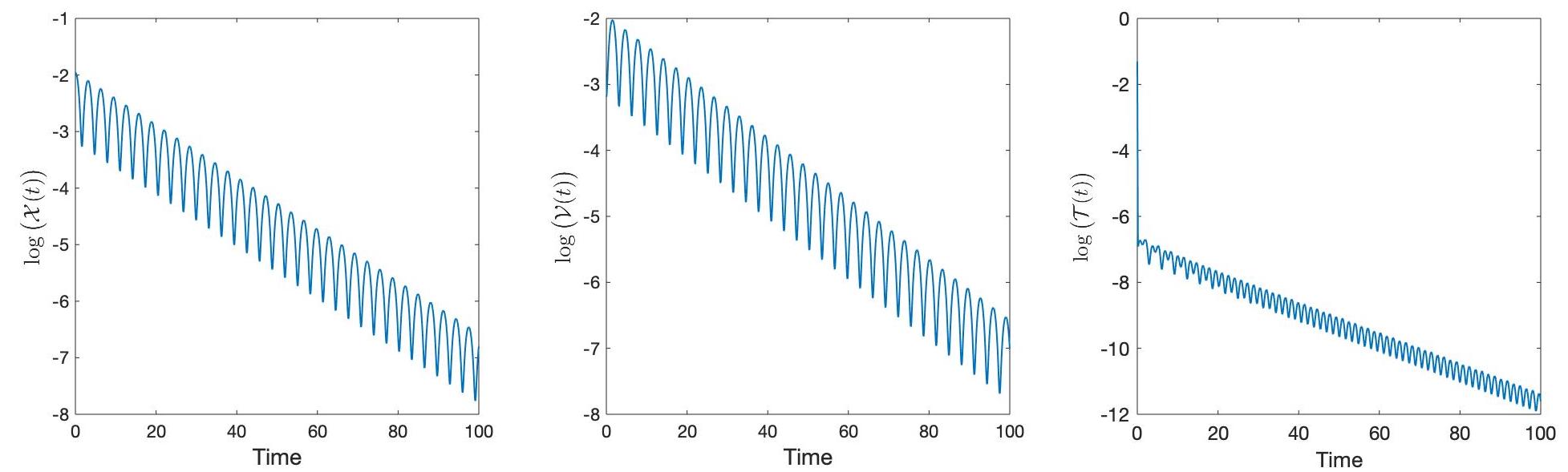} \caption{Decaying results for $\ell^2$ fluctuations $ \log\big(\mathcal{X}(t)\big)$, $ \log\big(\mathcal{V}(t)\big)$ and $\log\big(\mathcal T(t)\big)$.} \label{fig5}
\end{figure}

\section{ Conclusion} \label{sec:7}
  \setcounter{equation}{0}
   In this paper, we  studied emergent dynamics of the thermodynamic Cucker-Smale  model in a harmonic potential field, and provided asymptotic formation of periodically rotating one-point cluster which cannot be seen from the Cucker-Smale model in the absence of a harmonic potential field.  For the emergent dynamics, we need well-prepared initial data which are confined in a certain range  of the  state space. To guarantee the positivity of temperatures in a whole time interval, we first make an ansatz on the temperatures in a short-time interval to make sure the existence of  the solution for the system, and then obtain the  exponential decay for the fluctuations of  position and velocity from the dissipative differential inequality in the same short-time.  Using this result, one can deduce that  fluctuations of temperature decay exponentially  which improves the ansatz. By the continuity argument, we derived the formation of periodically rotating one-point cluster exponentially fast. Of course, our analytical framework is only a sufficient one, hence once the initial data  do not satisfy conditions in our proposed framework, then our results cannot say anything definite. As in the Cucker-Smale model, multi-clusters can emerge from the given initial data which do not satisfy our proposed framework. We leave this interesting issue for a future work. 
  
    \bibliography{references}

\begin{thebibliography}{10}

\bibitem{A-B-P}  J. A. Acebron, L. L. Bonilla, C. J. P\'{e}rez Vicente, F. Ritort and R. Spigler, \textit{The Kuramoto model: A simple paradigm for synchronization phenomena.} Rev. Mod. Phys. \textbf{77} (2005), 137-185.

\bibitem{A-B} G. Albi, N. Bellomo, L. Fermo, S.-Y. Ha, J. Kim,  L. Pareschi, D. Poyato  and J. Soler,  \textit{Vehicular traffic, crowds, and swarms: from kinetic theory and multiscale methods to applications and research perspectives.} Math. Models Methods Appl. Sci. {\bf29} (2019),  1901--2005.

\bibitem{AH} 
  S. Ahn and S.-Y Ha, \textit{Stochastic flocking dynamics of the Cucker-Smale model with multiplicative white noises}. J. Math. Phys. {\bf 51} (2010), 103301.
    
%

   \bibitem{B-B}  J. Buck and  E. Buck, \textit{Biology of synchronous flashing of fireflies.} Nature \textbf{211} (1966), 562-564.
     
      
          \bibitem{C-D}
      J. A. Carrillo, M. R. D'Orsogna and V. Panferov, \textit{Double milling in self-propelled swarms from kinetic theory}. Kinetic and Related Models {\bf 2} (2009), 363-378.
      
          \bibitem{CH1}
      Y.-P. Choi,  S.-Y. Ha, J. Jung and  J. Kim, \textit{Global dynamics of the thermomechanical Cucker-Smale ensemble immersed
in incompressible viscous fluids}. Nonlinearity {\bf 32} (2019), 1597-1640.
      
          \bibitem{CH2}
     Y.-P. Choi,  S.-Y. Ha, J. Jung and J. Kim, \textit{On the coupling of kinetic thermomechanical Cucker-Smale equation and compressible viscous  fluid system}. J. Math. Fluid
Mech. {\bf 22} (2020), 1597-1640.
      
       \bibitem{CS1}
            F. Cucker and  S. Smale, \textit{Emergent behavior in flocks}. IEEE Trans. Automat. Control {\bf52} (2007), 852-861.
          
                 \bibitem{CS2}
   F. Cucker and  S. Smale, \textit{On the mathematics of emergence}. Jpn. J. Math. {\bf 2} (2007), 197-227.
                 
                     \bibitem{D-M-1} 
            R.  Duan, M. Fornasier and G. Toscani, \textit{A kinetic flocking model with diffusion}. Comm.
Math. Phys. {\bf300} (2010), 95-145.
    \bibitem{D-M} 
   P.   Degond and  S. Motsch, \textit{Large-scale dynamics of the Persistent Turing Walker model of fish behavior}. J. Stat. Phys. {\bf 131} (2008), 989-1022.
     
    \bibitem{D-H-K} 
    J.-G. Dong,  S.-Y. Ha and D. Kim,  \textit{Emergent behaviors of continuous and discrete thermomechanical Cucker-Smale models on general digraphs}.  Math. Models Methods Appl. Sci. {\bf 29} (2019), 589-632. 
     

     \bibitem{E}
     R. Erban, J.  Ha\v{s}kovec and Y. Sun, \textit{A Cucker-Smale model with noise and delay}. SIAM J. Appl. Math. {\bf 76} (2016),1535-1557. 

       
              \bibitem{HL} 
  S.-Y. Ha, K.  Lee and D. Levy, \textit{Emergence of time-asymptotic flocking in a stochastic Cucker-Smale system}. Commun. Math. Sci. {\bf 7} (2009), 453-469.
       
        \bibitem{HH} 
        S.-Y. Ha, T. Ha and  J. Kim, \textit{Asymptotic flocking dynamics for the Cucker-Smale model with the Rayleigh
friction}. J. Phys. A: Math. Theor. {\bf 43} (2010), 31520.

    \bibitem{H-R} 
   S.-Y.  Ha and T. Ruggeri,  \textit{Emergent dynamics of a thermodynamically consistent particle model}. Arch. Ration. Mech. Anal. {\bf 223} (2017), 1397-1425. 
    
      \bibitem{HKR} 
       S.-Y. Ha,  J. Kim and T. Ruggeri, \textit{Emergent behaviors of thermodynamically consistent particle}. SIAM J. Math. Anal. {\bf 30} (2018), 3092-3121. 
      
      \bibitem{H-T} 
      S.-Y. Ha and  E. Tadmor, \textit{From particle to kinetic and hydrodynamic descriptions of flocking.} Kinet. Relat. Models {\bf 1} (2008), 415-435. 
            
%
      \bibitem{M-T-0} 
      S. Motsch  and  E. Tadmor, \textit{Heterophilious dynamics enhances consensus.} SIAM Rev. {\bf 56} (2014), 577--621.

      \bibitem{M-T} S. Motsch and E. Tadmor, \textit{A new model for self-organized dynamics and its flocking behavior.} J. Stat. Phys. {\bf 144} (2011), 923-947. 
      
      \bibitem{Pe} C. S. Peskin,  \textit{Mathematical aspects of heart physiology.} Courant Institute of Mathematical Sciences, New York, 1975.

      \bibitem{P-R-K} A. Pikovsky, M.  Rosenblum and J. Kurths, \textit{Synchronization: A universal concept in nonlinear sciences.} Cambridge University Press, Cambridge, 2001. 


      \bibitem{ST} R. Shu  and  E. Tadmor, \textit{Flocking Hydrodynamics with External Potentials}. Arch Rational Mech. Anal. {\bf 238} (2020), 347-381.

   \bibitem{T-T}  J. Toner and  Y.  Tu, \textit{Flocks, herds, and Schools: A quantitative theory of flocking.}
    Physical Review E {\bf 58} (1998), 4828-4858.
    
 \bibitem{T-B}  C. M. Topaz and A. L.  Bertozzi, \textit{Swarming patterns in a two-dimensional kinematic model for biological groups}. SIAM J. Appl. Math. {\bf 65} (2004), 152-174.
%

\bibitem{VZ} T. Vicsek and A. Zefeiris, \textit{Collective motion.} Phys. Rep. {\bf 517} (2012), 71-140.

\bibitem{Wi1} A. T. Winfree,  \textit{The geometry of biological time.} Springer, New York, 1980.

%

 .
\end{thebibliography}

     \end{document}